\documentclass[11pt]{article}
\usepackage[small]{titlesec}
\usepackage[cm]{fullpage}

\usepackage{amsmath,amscd, float,rotating}
\usepackage{pb-diagram}

\usepackage{latexsym}
\usepackage{amsfonts}
\usepackage{amsmath}
\usepackage{amssymb}

\numberwithin{equation}{section}
\newcommand{\qed}{\hfill \ensuremath{\Box}}

\def\XXint#1#2#3{{\setbox0=\hbox{$#1{#2#3}{\int}$}
\vcenter{\hbox{$#2#3$}}\kern-.5\wd0}}

\newcommand{\tr}[2]{\textrm{tr}_{#1} \, {#2}}
\newcommand{\ve}{\varepsilon}

\newcommand{\nablaR}{\nabla_{\mathbb{R}}}

\newcommand{\dbar}{\overline{\partial}}

\newcommand{\ddt}[1]{\frac{\partial #1}{\partial t}}

\newcommand{\ov}[1]{\overline{#1}}

\newcommand{\ddbar}{\frac{\sqrt{-1}}{2\pi} \partial\dbar}
\newcommand{\geuc}{g_{\textrm{Eucl}}}
\newcommand{\oeuc}{\omega_{\textrm{Eucl}}}

\begin{document}
\newtheorem{claim}{Claim}
\newtheorem{theorem}{Theorem}[section]
\newtheorem{proposition}{Proposition}[section]
\newtheorem{lemma}{Lemma}[section]
\newtheorem{definition}{Definition}[section]
\newtheorem{conjecture}{Conjecture}[section]
\newtheorem{corollary}{Corollary}[section]
\newtheorem{remark}{Remark}[section]
\newenvironment{proof}[1][Proof]{\begin{trivlist}
\item[\hskip \labelsep {\bfseries #1}]}{\end{trivlist}}
\newenvironment{example}[1][Example]{\addtocounter{remark}{1} \begin{trivlist}
\item[\hskip \labelsep {\bfseries #1
\thesection.\theremark}]}{\end{trivlist}}
~

\centerline{\bf \Large Contracting exceptional divisors by the K\"ahler-Ricci flow\footnote{The first-named author is supported in part by an NSF CAREER grant 
  DMS-08-47524  and the second-named author by the grant DMS-08-48193.  Both authors are also supported in part by Sloan Research Fellowships.}}

\bigskip
\bigskip

\centerline{\large \bf Jian Song$^{*}$ and
Ben Weinkove$^\dagger$}


\bigskip
\bigskip
\noindent
{\bf Abstract} \ We give a criterion under which a solution $g(t)$ of the K\"ahler-Ricci flow contracts exceptional divisors on a compact manifold and can be  uniquely continued   on a new manifold.  As   $t$ tends to the singular time $T$ from each direction,   we prove convergence of $g(t)$ in the sense of Gromov-Hausdorff  and smooth convergence away from the exceptional divisors.  We call this behavior for the K\"ahler-Ricci flow a \emph{canonical surgical contraction}.  In particular, our results show that the K\"ahler-Ricci flow on a projective algebraic surface will perform a sequence of canonical surgical contractions until, in finite time, either the minimal model is obtained, or the volume of the manifold tends to zero.

\bigskip


\section{Introduction}

In his seminal work \cite{H1}, Hamilton introduced the Ricci flow  and showed that it smoothly deforms a metric with positive Ricci curvature on a 3-manifold to a metric of constant curvature. In general, the Ricci flow on a 3-manifold may develop singularities.  Hamilton conjectured that the flow would break the manifold into pieces at the singular times, and initiated a program of `Ricci flow with surgeries' \cite{H2}.  The idea is 
to perform a surgery `by hand' on the manifold just before a singular time $T$ and then restart the Ricci flow on the new manifold.  Ultimately the goal was to  
 prove Thurston's geometrization conjecture for 3-manifolds using the Ricci flow.  In 2002, Perelman's ground-breaking work \cite{P1} refined Hamilton's surgery process and, according to  \cite{CZ, KL, MT}, resolved all of the key difficulties in the  geometrization program.

The method of Ricci flow with surgery appears to be very powerful.  However,  the surgery process is not canonical.  It has been suggested (in, for example, \cite{P1})  that the Ricci flow should carry out surgeries through the singularities in some natural and unique way.  This is referred to as \emph{canonical surgery by the Ricci flow}.  As $t \rightarrow T$, does the flow $(X, g(t))$ converge in some suitable sense to a  `limit manifold'  $(\ov{X}, g_T)$ as $t$ tends to the singular time $T$?   Is there then a unique flow on the new manifold starting at $g_T$?

There has been much interest in studying this phenomenon in the case of K\"ahler manifolds (see for example \cite{FIK}).  The Ricci flow preserves the K\"ahler condition, and the additional structure in the K\"ahler setting makes the problem of formulating and understanding canonical surgery more tractable.  Moreover, if the manifold $X$ is algebraic then one could hope that the canonical surgeries correspond to algebraic operations, thus opening the way for the use of the Ricci flow to prove results in algebraic geometry.

In the 1980's,  Tsuji \cite{Ts} applied the  K\"ahler-Ricci flow on projective varieties with nef canonical bundle  to construct K\"ahler-Einstein metrics with singularities.  More recently,   the  singular 
behavior of the K\"ahler-Ricci flow has been investigated further \cite{TZha, Zha, SoT1, SoT2}.  Moreover, it has been conjectured  that a suitable notion of the K\"ahler-Ricci flow with surgery will carry out the minimal model program in algebraic geometry \cite{T2, LT, SoT3}.   That is, given an algebraic variety $X$, one hopes to obtain a unique model variety within its birational class, called the `minimal model', which is obtained from $X$ by certain algebraic operations (for some exciting recent developments, see \cite{BCHM, Si}).  Conjecturally, the K\"ahler-Ricci flow will carry out these algebraic operations in the form of canonical surgeries and eventually arrive at the minimal model.

A detailed program is laid out in [SoT3] for how the K\"ahler-Ricci flow will behave on a general projective algebraic variety.  More precisely, it is conjectured that the K\"ahler-Ricci flow will either
 deform a projective algebraic variety $X$ to its minimal model via finitely many divisorial contractions and flips in the Gromov-Hausdorff sense, then eventually converge to  a generalized K\"ahler-Einstein metric on the canonical model of $X$ (after normalization), or collapse in finite time. This is  the analytic analogue of Mori's minimal model program. The existence and uniqueness is proved in \cite{SoT3} for the weak solution of the K\"ahler-Ricci flow through divisorial contractions and flips.  However, the Gromov-Hausdorff convergence at the singular time is largely open.  The current paper establishes  the above conjecture when $X$ is a projective algebraic surface.  More generally, we deal with the case of blowing-down exceptional divisors by the K\"ahler-Ricci flow.

We illustrate with a special case  when the manifold $X$ is $\mathbb{P}^n$ blown up at one point.  This is  projective space with a point $y_0$ replaced by a subvariety $E$, the \emph{exceptional divisor}, which is biholomorphic to $\mathbb{P}^{n-1}$ and represents all of the directions through $y_0$.   Under an assumption on the initial K\"ahler class (see (\ref{a0b0}) below), Feldman-Ilmanen-Knopf  \cite{FIK} conjectured  that the K\"ahler-Ricci flow should `contract'  $E$ at the first singular time $T$ and that the manifolds should converge in the Gromov-Hausdorff sense to $\mathbb{P}^n$.  Moreover, they constructed self-similar solutions of the K\"ahler-Ricci flow which exhibit precisely this behavior.  The paper \cite{SW1} established some conjectures of \cite{FIK},  including this contracting of the exceptional divisor,  under a symmetry assumption which reduces the equation to a parabolic PDE in one space variable.   In the current paper we make no symmetry assumptions.  Before we state our results in full generality, we describe informally what we can prove in this special case of $\mathbb{P}^n$ blown up at one point.

Write $\pi: X \rightarrow \mathbb{P}^n$ for the blow-down map, which is an isomorphism from $X \setminus E$ to $\mathbb{P}^n \setminus \{ y_0 \}$.  As $t$ tends to the first singular time $T$, the metrics $g(t)$ converge smoothly on compact subsets of $X \setminus E$ to a smooth K\"ahler metric $g_T$, defined on $X \setminus E$ (this part was already known - see \cite{TZha} and also \cite{SW1}).  Using the map $\pi$,  $g_T$ defines a K\"ahler metric on $\mathbb{P}^n \setminus \{ y_0 \}$.  By extending this tensor to be zero at $y_0$, we  define a distance function $d_T$ on $\mathbb{P}^n$ and show that  the metric space $(\mathbb{P}^n, d_T)$  is homeomorphic to the manifold $\mathbb{P}^n$.
We prove that $(X, g(t))$ converges to $(\mathbb{P}^n, d_T)$ in the Gromov-Hausdorff sense as $t \rightarrow T^-$.  Next, using results of \cite{SoT3}, we show that there is a unique smooth solution $g(t)$ of the K\"ahler-Ricci flow on $\mathbb{P}^n$ for $t$ in $(T, T+\delta)$ for some $\delta>0$ such that $g(t)$ converges to $(\pi^{-1})^* g_T$ smoothly on compact subsets of $\mathbb{P}^n \setminus \{ y_0\}$ as $t \rightarrow T^+$.  Finally, we show that $(\mathbb{P}^n, g(t))$ converges to $(\mathbb{P}^n, d_T)$ in the Gromov-Hausdorff sense as $t \rightarrow T^+$.  Thus in both cases $t\rightarrow T^-$ and $t\rightarrow T^+$ we have smooth convergence away from the exceptional divisor as well as global Gromov-Hausdorff convergence.  We call this behavior for the K\"ahler-Ricci flow a  \emph{canonical surgical contraction}.   We now define this more precisely.

Let $(X, g_0)$ be a compact K\"ahler manifold of complex dimension $n \ge 2$. We write $\omega_0= \frac{\sqrt{-1}}{2\pi} (g_0)_{i \ov{j}}dz^i \wedge d\ov{z^j}$ for the K\"ahler form associated to $g_0$.  We consider the K\"ahler-Ricci flow $\omega=\omega(t)$ given by
\begin{equation} \label{krf0}
\frac{\partial}{\partial t} \omega = - \textrm{Ric}(\omega), \quad \omega|_{t=0} = \omega_0,
\end{equation}
for $ \textrm{Ric}(\omega) = - \ddbar \log \det g$, where $g=g(t)$ is the metric associated to $\omega$.  As long as the flow exists, the K\"ahler class of $\omega(t)$ is given by
\begin{equation}
[\omega(t)] = [\omega_0] + t  c_1(K_X) >0,
\end{equation}
where $K_X$ is the canonical bundle of $X$.   The first singular time  $T$ is characterized by
\begin{equation} \label{T}
T = \sup \{ t \in \mathbb{R} \ | \ [\omega_0] + t  c_1(K_X) >0 \}.
\end{equation}
Clearly the number $T$ depends  only on $X$ and the K\"ahler class $[\omega_0]$, and satisfies  $0 < T \le \infty$.
 It was shown in \cite{TZha} that there exists a smooth solution of the K\"ahler-Ricci flow (\ref{krf0}) for $t$ in $[0,T)$.  If $T$ is finite then the flow must develop a singularity as $t \rightarrow T^-$.

\begin{definition} \label{defn}
We say that the solution $g(t)$ of the K\"ahler-Ricci flow, as above, performs a  {\bf  canonical surgical contraction} if the following holds.  There exist distinct analytic subvarieties $D_1, \ldots, D_k$ of $X$ of codimension 1, a compact K\"ahler manifold $Y$ and a surjective holomorphic map $\pi: X \rightarrow Y$ with  $\pi(D_i) = y_i \in Y$ and
$\pi|_{X \setminus \bigcup_{i=1}^k D_i}$  a biholomorphism onto $Y \setminus \{ y_1, \ldots, y_k \}$ such that:
\begin{enumerate}
\item[(i)]  As $t \rightarrow T^-$,   the metrics $g(t)$ converge to a smooth K\"ahler metric $g_T$ on $X \setminus \bigcup_{i=1}^k D_i$ smoothly on compact subsets of $X \setminus \bigcup_{i=1}^k D_i.$
\item[(ii)]  Let $d_T$ be the metric on $Y$ defined by extending $(\pi^{-1})^*g_T$ to be zero on $\{ y_1, \ldots, y_k \}$ (see Definition \ref{defndy}).  Then $(Y, d_T)$ is a compact metric space homeomorphic to the manifold $Y$.
\item[(iii)] $(X, g(t))$ converges to $(Y, d_T)$ in the Gromov-Hausdorff sense as $t\rightarrow T^-$.
\item[(iv)] There exists a unique maximal smooth solution $g(t)$ of the K\"ahler-Ricci flow on $Y$ for $t\in (T, T_Y)$, with $T< T_Y \le \infty$, such that $g(t)$ converges to $(\pi^{-1})^*g_T$ as $t  \rightarrow T^+$ smoothly on compact subsets of $Y \setminus \{ y_1, \ldots, y_k\}$.  Moreover $g(t)$ for $t \in [0, T)$ is smoothly connected via $\pi$ to $g(t)$ for $t \in [T, T_Y)$ outside the points $y_1, \ldots, y_k$ (see Section \ref{sectionafterT} for a precise definition).
\item[(v)] $(Y, g(t))$ converges to $(Y, d_T)$ in the Gromov-Hausdorff sense as $t\rightarrow T^+$.
\end{enumerate}
\end{definition}

Definition \ref{defn} is a special case of the K\"ahler-Ricci flow with surgery discussed in \cite{SoT3} (see also \cite{FIK}).
Note that the notion of convergence in this definition is stronger than Gromov-Hausdorff convergence.  In particular, no diffeomorphisms are involved.  As an aside, we remark that the definition of canonical surgical contraction can be generalized to allow $X$ and $Y$ to have mild singularities, and further to the case when a flip occurs.

From Definition \ref{defn}, canonical surgical contractions are unique up to biholomorphisms if they exist.   More precisely, suppose a canonical surgical contraction  exists for a solution $g(t)$ of the K\"ahler-Ricci flow on $[0,T)$ with respect to $D_1, \ldots, D_k$, $Y$ and $\pi: X \rightarrow Y$,  and that  there also exists one with respect to $\tilde{D}_1, \ldots, \tilde{D}_{\tilde{k}}$, $\tilde{Y}$ and $\tilde{\pi}: X \rightarrow \tilde{Y}$.  Then $\tilde{k}=k$ and (after possibly reordering) $\tilde{D}_i=D_i$ for $i=1, \ldots, k$.  Moreover, there exists a biholomorphism $F : Y \rightarrow \tilde{Y}$ such that $\tilde{\pi} = F \circ \pi$.  For $t>T$ we have $g(t) = F^* \tilde{g}(t)$.

Our main result is:

\begin{theorem} \label{thmblowup} Let $g(t)$ be a smooth solution of the K\"ahler-Ricci flow (\ref{krf0}) on $X$ for $t$ in $[0,T)$ and  assume $T<\infty$. Suppose there exists a blow-down map $\pi:  X \rightarrow Y$ contracting disjoint irreducible exceptional divisors $E_1, \ldots, E_k$ on $X$ with $\pi(E_i) = y_i \in Y$, for a smooth compact K\"ahler manifold $(Y, \omega_Y)$   such that the limiting K\"ahler class satisfies
\begin{equation} \label{assumption1}
[\omega_0] + T c_1(K_X) = [\pi^*\omega_Y].
\end{equation}
Then the K\"ahler-Ricci flow $g(t)$ performs a canonical surgical contraction with respect to the data $E_1, \ldots, E_k$, $Y$ and $\pi$.
\end{theorem}

This theorem generalizes some of the results in \cite{SW1} which considered the case of certain manifolds with symmetry.  A version of 
 Theorem \ref{thmblowup}, broadly interpreted, was conjectured by Feldman-Ilmanen-Knopf \cite{FIK}.
 
 \begin{remark} \label{remark1}
We conjecture that, in addition, $\pi|_{X\setminus \bigcup_{i=1}^k E_i}$ extends to an isometry from the metric completion of $(X \setminus \bigcup_{i=1}^k E_i, g_T)$ to $(Y, d_T)$, so that a stronger notion of canonical surgical contraction holds under the hypotheses of Theorem \ref{thmblowup}.   In an earlier version of this paper we erroneously claimed this stronger result.  We  are grateful to Yuguang Zhang for pointing out a gap in our previous argument.
\end{remark}
 
We  remark that  if the K\"ahler-Ricci flow $g(t)$ performs a canonical surgical contraction then there exists a uniform constant $C$ such that for $t \in [0,T)$,
\begin{equation}
\textrm{diam}_{g(t)} X \le C.
\end{equation}
Indeed, this is an immediate consequence of part (iii) of the definition.

Returning to the example of $\mathbb{P}^n$ blown up at one point, we apply our theorem to obtain:

\begin{example} \label{ex} Denote by $X$ the blow-up of $\mathbb{P}^n$ at one point $y_0$.  Let  $\pi: X \rightarrow \mathbb{P}^n$ be the blow-down map, which sends the exceptional divisor $E$ to $y_0 \in \mathbb{P}^n$.  K\"ahler classes on $X$ can be written as
\begin{equation}
\alpha = b\, \pi^*[H] - a [E], \qquad \textrm{for } 0<a<b,
\end{equation}
where $H$ is a hyperplane in $\mathbb{P}^n$.  We consider a solution of the K\"ahler-Ricci flow (\ref{krf0}) starting at $\omega_0$ in a K\"ahler class $\alpha_0 = b_0 \, \pi^*[H] - a_0 [E]$ where $a_0$ and $b_0$ satisfy the condition:
\begin{equation} \label{a0b0}
a_0(n+1) < b_0(n-1).
\end{equation}
One can compute that $T=a_0/(n-1)$ and
\begin{equation}
[\omega_0]+ T c_1(K_X) = [\kappa \pi^*\omega_{\textrm{FS}}], \qquad  \textrm{for }\  \kappa = \frac{(n-1)b_0-(n+1)a_0}{n-1} >0,
\end{equation}
where $\omega_{\textrm{FS}}$ is the Fubini-Study metric on $\mathbb{P}^n$.  Hence Theorem \ref{thmblowup} can be applied and the K\"ahler-Ricci flow performs a canonical surgical contraction at time $T$ with respect to $E$, $\mathbb{P}^n$ and $\pi$.
\end{example}

Once a canonical surgical contraction has been performed on $[0,T)$,
we have a unique smooth solution of the K\"ahler-Ricci flow on the manifold $Y$, which exists for a maximal time interval $(T, T_1)$ say.  Repeating this process, we can ask whether the flow performs a contraction on $Y$.

We define a \emph{maximal K\"ahler-Ricci flow with canonical surgical contractions} starting from $(X, \omega_0)$ to be a sequence of $k+1$ smooth solutions of the K\"ahler-Ricci flow, each written $g(t)$, on manifolds $X_0=X, X_1, X_2, \ldots, X_k$ on maximal time intervals $[0,T_0)$, $(T_0, T_1), \ldots, (T_{k-1}, T_k)$  such that $g(t)$ performs a canonical surgical contraction at times $T_0, T_1, \ldots, T_{k-1}$ but not at time $T_k$.  Note that we allow $T_k=\infty$ as a possibility.  By definition, if such a flow exists starting from $(X, \omega_0)$ then it is unique.

In the case when $X$ is an algebraic surface and $\omega_0$ is in a rational class, we see that such a maximal flow always exists and indeed can be used to give information about the original manifold $X$.  In this case, exceptional divisors are referred to as exceptional curves (or exceptional curves of the first kind).  For definitions of some more of the algebraic geometry terminology see Section \ref{sectionapply}.

\begin{theorem} \label{theoremclass}
Let $X$ be a projective algebraic surface and $\omega_0$ a K\"ahler metric with $[\omega_0] \in H^{1,1}(X, \mathbb{Q})$.  Then there exists a unique maximal K\"ahler-Ricci flow $\omega(t)$ on $X_0, X_1, \ldots, X_k$ with canonical surgical contractions  starting at $(X, \omega_0)$. Moreover, each canonical surgical contraction  corresponds to a blow-down $\pi: X_i \rightarrow X_{i+1}$ of a finite number of disjoint exceptional curves on $X_i$.  In addition we have:
\begin{enumerate}
\item[(i)]  Either $T_k < \infty$ and the flow $\omega(t)$ collapses $X_k$, in the sense that
$$\emph{Vol}_{\omega(t)} X_k \rightarrow 0, \quad \textrm{as} \ \ t \rightarrow T_k^-.$$
In this case $X_k$ is a Fano surface or a ruled surface.
\item[(ii)]  Or $T_k = \infty$ and $X_k$ has no exceptional curves of the first kind.
\end{enumerate}
\end{theorem}

Observe that in case (i), by well-known results in algebraic geometry, $X$ has negative Kodaira dimension,  and thus is either $\mathbb{P}^2$ or a ruled surface. 

It was shown in \cite{So} that the K\"ahler-Ricci flow shrinks to a point if and only if $X$ is Fano and the initial K\"ahler class is proportional to $c_1(X)$, establishing a special case of a conjecture in \cite{T2}.  If this occurs, it is natural to renormalize the flow so that the volume is constant.  
The problem of how  this normalized  flow behaves  is of great interest and is still open in general.  Assuming the existence of a K\"ahler-Einstein metric \cite{P2,TZhu} or soliton \cite{TZhu}, the flow has been shown to converge to a K\"ahler-Einstein metric or soliton respectively (see also \cite{SeT, Zhu}).  The existence of K\"ahler-Einstein metrics  on Fano manifolds is conjecturally related to stability \cite{Y2,T2,Do}.  The connection between stability conditions and the behavior of the K\"ahler-Ricci flow has been studied in \cite{PS, PSSW1, PSSW2, R, Sz, To, MS, CW} for example.

In case (ii),  $X$ must have   nonnegative Kodaira dimension (see Remark \ref{negkod}.1). Thus $X_k$ is the minimal model of $X$ (the unique algebraic surface birational to $X$ which has no exceptional curves of the first kind).  We remark that although the final minimal surface is unique, the order of the canonical surgical contractions  does depend on the initial K\"ahler class. 

 In fact, we can say more in case (ii).  There are  three distinct behaviors of the K\"ahler-Ricci flow as $t \rightarrow \infty$ depending on whether 
 whether $X$ has Kodaira dimension, written $\textrm{Kod}(X)$, equal to 0, 1 or 2:
\begin{itemize}
\item if $\textrm{Kod}(X)=0$, then the minimal model of $X$ is a Calabi-Yau surface. The flow $g(t)$ converges smoothly to a Ricci-flat K\"ahler metric as $t\rightarrow \infty$ \cite{Y1, Cao}.
\item If $\textrm{Kod}(X)=1$, then 
 $\frac{1}{t} \omega(t)$ converges in the sense of currents   to the pullback of the unique generalized K\"ahler-Einstein metric on the canonical model     of $X$ as $t\rightarrow \infty$ \cite{SoT1}.
\item If $\textrm{Kod}(X)=2$, $\frac{1}{t} \omega(t)$ converges in the sense of currents (and smoothly outside a subvariety) to the pullback of the unique smooth orbifold K\"ahler-Einstein metric on the canonical model of $X$ as $t\rightarrow \infty$ \cite{Ts, TZha}. 
\end{itemize}
Hence the K\"ahler-Ricci flow appears to give a rough classification of algebraic surfaces.   For precise definitions and futher details we refer the readers to \cite{T2} and the papers listed above. 

To generalize these kinds of results to higher dimensions requires an understanding of the behavior of the K\"ahler-Ricci flow through a divisorial flip.  In the work of La Nave-Tian \cite{LT} the conjectural behavior of the flow through a flip is discussed in relation to their  V-soliton equation\footnote{The authors have been informed by G. Tian that he and G. La Nave are currently writing up some further results on the V-soliton equation which can treat the case of Theorem \ref{thmblowup}.}.

We remark that our techniques can also be used in some other situations, once we extend our definitions to deal with possible orbifold singularities.
  Using the notation of \cite{SW1}, let $X$ be the manifold $M_{n,k}$  with  $1 \le k \le n-1$. The manifolds $M_{n,k}$ are $\mathbb{P}^1$ bundles over $\mathbb{P}^{n-1}$ which generalize the Hirzebruch surfaces $M_{2,k}$ (see \cite{C1}, for example).   Under a necessary condition on the initial K\"ahler class,  the K\"ahler-Ricci flow will perform a canonical surgical contraction of the `zero' divisor $D_0$ in $M_{n,k}$ (if $k=1$ this coincides with Example \ref{ex}.1, with $D_0$ the exceptional divisor).    
If the initial metric satisfies a certain symmetry condition the behavior of the flow as $t \rightarrow T^-$ was dealt with in  \cite{SW1}.

   In addition, our methods  can be used to show that if $X$ is a minimal surface of general type with  one irreducible $(-2)$ curve then the solution $(X, g(t))$ of the normalized K\"ahler-Ricci flow will converge in the Gromov-Hausdorff sense to the canonical model of $X$ with its K\"ahler-Einstein metric.  These results will appear in a forthcoming work.

We now give a brief outline of the current paper.  In Section \ref{sectex} we prove the main estimates for the metric $g(t)$ as $t\rightarrow T^-$ along the flow under the assumptions of Theorem \ref{thmblowup}.  In Section \ref{sectionGH} we prove the Gromov-Hausdorff convergence of $(X,g(t))$ as $t \rightarrow T^-$, establishing (ii) and (iii) in the definition of canonical surgical contraction.  In Section \ref{sectionhigher} we prove higher order estimates for $g(t)$ as $t \rightarrow T^-$, which are needed for continuing the flow after time $T$.  In Sections
 \ref{sectionafterT} and \ref{sectionproof} we complete the proof of Theorem \ref{thmblowup}.  Finally, in Section \ref{sectionapply} we prove Theorem \ref{theoremclass}.

A remark about notation:  we will often use $C$ or $C'$ to denote a uniform constant, which may differ from line to line.

\setcounter{equation}{0}
\section{Key estimates} \label{sectex}

In this section we prove our main estimates for the K\"ahler-Ricci flow under the assumptions of Theorem \ref{thmblowup}. 
Since the $E_1, \ldots, E_k$ do not intersect each other, we may assume for simplicity that $k=1$ (it will be clear that the arguments below generalize to the case of $k$ non-intersecting divisors).
Thus the map $\pi$  contracts exactly one exceptional divisor, which we denote by $E$.  We write the image of $E$ under $\pi$ as $y_0: =\pi(E) \in Y$.

We write the K\"ahler-Ricci flow (\ref{krf0}) as a parabolic complex Monge-Amp\`ere equation.  First, using assumption (\ref{assumption1}),  define a family of reference metrics $\hat{\omega}_t$ for $t \in [0,T)$ by
\begin{equation}
\hat{\omega}_t = \frac{1}{T} \left( (T-t) \omega_0 + t \pi^* \omega_Y \right) \in [\omega(t)] =  [\omega_0]+t c_1(K_X).
\end{equation}
 Now let $\Omega$ be the unique volume form on $X$ satisfying
 \begin{equation}
 \frac{\sqrt{-1}}{2\pi} \partial \ov{\partial} \log \Omega = \ddt{} \hat{\omega}_t = \frac{1}{T}( \pi^* {\omega}_Y - \omega_0) \in c_1(K_X), \quad \int_X \Omega =1.
 \end{equation}

It follows that if $\varphi=\varphi(t)$ solves the parabolic complex Monge-Amp\`ere equation
\begin{equation} \label{KRF2}
\ddt{\varphi} = \log \frac{ (\hat{\omega}_t + \frac{\sqrt{-1}}{2\pi} \partial \ov{\partial} \varphi )^n}{\Omega}, \quad \varphi|_{t=0} =0,
\end{equation}
for $t$ in $[0,T)$ then $\omega = \hat{\omega}_t + \frac{\sqrt{-1}}{2\pi} \partial \ov{\partial} \varphi$ solves (\ref{krf0}).  Conversely, given the solution $\omega=\omega(t)$ of $(\ref{krf0})$ on $[0,T)$ one can obtain, after normalization, the unique solution $\varphi=\varphi(t)$ of (\ref{KRF2}).

In the following, we assume the hypotheses of Theorem \ref{thmblowup}.  The estimates of the first lemma are essentially contained in \cite{TZha, Zha} (see also \cite{SW1}), but we give here a somewhat more elementary proof (see also \cite{T2}).

\begin{lemma} \label{Linfinity}
There is a uniform constant $C$ depending only on $(X, \omega_0)$  such that the solution $\varphi=\varphi(t)$ of (\ref{KRF2}) satisfies, for $t \in [0,T)$,
\begin{enumerate}
\item[(i)] $\displaystyle{\| \varphi \|_{L^{\infty}} \le C.}$
\item[(ii)] $\displaystyle{ \dot{\varphi} \le C},$ where we are writing $\dot{\varphi}$ for  $\partial \varphi/\partial t$.
\item[(iii)] $\displaystyle{ \omega^n \le C \Omega}$.
\item[(iv)] As $t \rightarrow T^-$, $\varphi(t)$ converges pointwise on $X$ to a bounded function $\varphi_T$ satisfying
\begin{equation}
\omega_T:= \hat{\omega}_T + \frac{\sqrt{-1}}{2\pi} \partial \ov{\partial} \varphi_T \ge 0,
\end{equation}
and $\omega(t)$ converges weakly in the sense of currents to the closed positive (1,1) current $\omega_T$.
\end{enumerate}
\end{lemma}
\begin{proof}
Two proofs of the uniform bound of $\varphi$ are given in \cite{TZha}  (see also \cite{Zha}).  The first uses a generalization of the deep result of Kolodziej \cite{Kol1} and the second uses the maximum principle.  We give here a different maximum principle argument.
 For the upper bound of $\varphi$, observe that there exists a uniform constant $C_0>0$ such that
$$\hat{\omega}_t^n \le C_0 \Omega.$$
Define  $\psi=\varphi-(\log C_0+1)t$ and suppose that there exists $(x_0, t_0) \in X \times (0,T)$ with
$\psi(x_0, t_0) = \sup_{X \times [0,t_0]} \psi.$
At $(x_0,t_0)$ we have $\sqrt{-1} \partial \ov{\partial}\varphi \le 0$ and
\begin{eqnarray*}
0 \le \ddt{\psi} \le \log \frac{\hat{\omega}_t^n}{\Omega} - \log C_0 -1 \le -1,
\end{eqnarray*}
a contradiction.  Hence $\psi$ attains its maximum at $t=0$ and it follows that $\psi$ and thus $\varphi$ are uniformly bounded from above for $t \in [0,T)$.

For the lower bound of $\varphi$ observe  that since $\pi^* \omega_Y \ge 0$ we have
$$\hat{\omega}_t^n \ge \frac{(T-t)^n}{T^n} \omega_0^n \ge c_0 (T-t)^n \Omega,$$
for a uniform constant $c_0>0$.  Now compute the evolution of the function $$\theta = \varphi + n(T-t)(\log (T-t)-1) - (\log c_0 - 1)t.$$
If there exists a point $(x_0, t_0) \in X \times (0,T)$ with $\theta(x_0, t_0) = \inf_{X \times [0,t_0]} \theta$ we have
\begin{eqnarray*}
0 \ge \ddt{\theta} \ge \log \frac{\hat{\omega}_t^n}{\Omega} - n \log (T-t) - \log c_0 +1 \ge 1,
\end{eqnarray*}
a contradiction.  Hence there exists a uniform constant $C$ with
$\theta \ge -C$.  Since the function $(T-t) \log (T-t)$ is uniformly bounded for $t\in [0,T)$ this shows that $\varphi$ is bounded from below and completes the proof of (i).

For (ii), we use the argument of \cite{TZha, Zha} (see also Theorem 3.2 in \cite{SW1}).  Writing $\eta = \frac{1}{T}(\pi^*\omega_Y - \omega_0)$ and $\Delta$ for the Laplacian with respect to $\omega$ we see that
 \begin{equation}
 (\ddt{} - \Delta) \dot{\varphi} = \tr{\omega}{\eta},
 \end{equation}
 where by definition $\tr{\omega}{\eta}= n (\eta \wedge \omega^{n-1})/ \omega^n$.
 Then
 \begin{equation}
 (\ddt{} - \Delta)(t \dot{\varphi}- \varphi -nt) =  - \textrm{tr}_{\omega} \omega_0 \le 0,
 \end{equation}
 using the fact that $\Delta \varphi = n- \textrm{tr}_{\omega}(\omega_0 + t \eta)$.
 It follows from the maximum principle that $t \dot{\varphi}$ is uniformly bounded from above.    Hence $\dot{\varphi}$ is uniformly bounded from above.

 The inequality (iii) follows immediately from (ii).  Part (iv) follows from (i) and (ii).
 \qed
\end{proof}

We can now show that we have $C^{\infty}$ estimates on $\omega(t)$ on compact subsets of $X \setminus E$   thus establishing part (i) in the definition of canonical surgical contraction.  This result is contained in \cite{TZha, Zha} (cf. \cite{Ts}), but we include a proof for the reader's convenience.

\begin{lemma} \label{lemmack}  With the assumptions of Theorem \ref{thmblowup}, the solution $\omega=\omega(t)$ of the K\"ahler-Ricci flow satisfies the following.
\begin{enumerate}
\item[(i)] There exists a uniform constant $c>0$ such that
\begin{equation} \label{uniformmetric}
\omega \ge c \pi^* \omega_Y.
\end{equation}
\item[(ii)] For every compact set $K \subset X \setminus E$, there exist constants $C_{K,i}$ for $i=0, 1, 2, \ldots$, such that
\begin{equation}
\| \omega \|_{C^i(K)} \le C_{K,i}.
\end{equation}
\item[(iii)]  The closed (1,1) current $\omega_T$, given by Lemma \ref{Linfinity}, is a smooth K\"ahler form on $X \setminus E$.
\item[(iv)]  As $t \rightarrow T^-$, the metrics $\omega(t)$ converge to $\omega_T$ in $C^{\infty}$ on compact subsets of $X \setminus E$.
\end{enumerate}

\end{lemma}
\begin{proof}
Part (i) follows from the parabolic Schwarz lemma \cite{SoT1}.  Indeed, if $w = \tr{\omega} \pi^* \omega_Y$ then (see for example the proof of Lemma 3.2 of \cite{SW1}),
\begin{equation}
\left( \ddt{} - \Delta \right) \log w \le C_0 w,
\end{equation}
for a constant $C_0$ depending on the curvature tensor of $\omega_Y$.  Choose a constant $A$ large enough so that $(A-1) \hat{\omega}_t \ge (C_0+1)\pi^* \omega_Y$ all $t\in [0,T)$.  
If $Q= \log w- A \varphi - An(T-t) ( \log (T-t)-1)$ then
\begin{eqnarray} \nonumber
\left( \ddt{} - \Delta \right) Q & \le & C_0 w - A \dot{\varphi} + An + An\log (T-t) - A\tr{\omega}{\hat{\omega}_t} \\
& \le & - w + A \log \left( \frac{(T-t)^n \Omega}{\omega^n} \right) +An - c_0  \left(\frac{(T-t)^n \Omega}{\omega^n}\right)^{1/n},
\end{eqnarray}
for a constant $c_0>0$ chosen so that 
\begin{equation}
\tr{\omega}{\hat{\omega}_t} \ge \left( \frac{T-t}{T} \right) \tr{\omega}{\omega_0} \ge  c_0 \left( \frac{ (T-t)^n\Omega}{\omega^n} \right)^{1/n}.
\end{equation}
   Since the function $\mu \mapsto A \log \mu - c_0 \mu^{1/n}$ for $\mu>0$ is bounded from above, a maximum principle argument shows that $Q$ and hence $w$ is uniformly bounded from above, giving (i).

For part (ii) of the lemma, note that on any compact set $K \subset X \setminus E$,  the (1,1) form $\pi^* \omega_Y$ is uniformly equivalent to $\omega_0$ (depending  on $K$).  From Lemma \ref{Linfinity}.(iii) we have the  bound $\omega^n \le C \Omega$.  Combining this with  (\ref{uniformmetric}), we see that $\omega$ is uniformly bounded from above on a fixed compact  subset $K \subset X \setminus E$.
The desired estimates then follow from the standard local parabolic theory.

Parts (iii) and (iv) of the lemma follows from parts (i) and (ii) together with Lemma \ref{Linfinity}.
\qed
\end{proof}

We now use the description of the blow-up map $\pi: X \rightarrow Y$ near $E$ to compare $\pi^* \omega_Y$ with a fixed K\"ahler metric $\omega_0$ on $X$.  Recall that we write $\pi(E)=y_0$.  It will be  convenient for us to fix, once and for all, a coordinate chart $U$ centered at $y_0$, which we identify via coordinates $z^1, \ldots, z^n$ with the unit ball $D$ in $\mathbb{C}^n$,
\begin{equation}
D= \{ (z^1,\ldots, z^n) \in \mathbb{C}^n \ | \ \sum_{i=1}^n |z^i|^2 <1 \}.
\end{equation}
Denote by $g_{\textrm{Eucl}}$ the Euclidean metric on $D$.  Since $g_{\textrm{Eucl}}$ and $g_{Y}$ are uniformly equivalent on $D$, we will see in the following that it often suffices to prove estimates for $g_{\textrm{Eucl}}$ instead of $g_Y$.  Write $D_r \subset D$ for the ball of radius $0<r<1$ with respect to $g_{\textrm{Eucl}}$.

We recall  the definition of the blow-up construction, following the exposition in \cite{GH}.  We identify $\pi^{-1}(D)$  with the submanifold $\tilde{D}$ of $D \times \mathbb{P}^{n-1}$ given by
\begin{equation} \label{Dtilde}
\tilde{D} = \{ (z,l) \in D \times \mathbb{P}^{n-1} \ | \ z^i l^j = z^j l^i \},
\end{equation}
where $l=[l^1, \ldots, l^n]$ are homogeneous coordinates on $\mathbb{P}^{n-1}$.  The map $\pi$ restricted to $\tilde{D}$ is the projection $\pi|_{\tilde{D}}(z,l) = z \in D$, with the exceptional divisor $E \cong \mathbb{P}^{n-1}$ given by $\pi^{-1}(0)$.   The map $\pi$ gives an isomorphism from  $\tilde{D} \setminus E$  onto the punctured ball $D\setminus \{0\}$.

On $\tilde{D}$ we have coordinate charts $\tilde{D}_i = \{ l^i \neq 0 \}$ with local coordinates $\tilde{z}(i)^1, \ldots, \tilde{z}(i)^n$ given by $\tilde{z}(i)^j = l^j/l^i = z^j/z^i$ for $j \neq i$ and $\tilde{z}(i)^i = z^i$.  The divisor $E$ is given in $\tilde{D}_i$ by $\{ \tilde{z}(i)^i=0 \}$.  The line bundle $[E]$ over $\tilde{D}$ has transition functions $z^i/z^j$ on $\tilde{D}_i \cap \tilde{D}_j$.
We can define a global section $s$ of $[E]$ over $X$ by setting $s(z) = z^i$ on $\tilde{D}_i$ and $s=1$ on $X \setminus \pi^{-1}(D_{1/2})$.  The section $s$ vanishes along the exceptional divisor $E$.
We also define a Hermitian metric $h$ on $[E]$ as follows.   First let  $h_1$ be the Hermitian metric on $[E]$ over $\tilde{D}$ given in $\tilde{D}_i$ by
\begin{equation}
h_1 = \frac{ \sum_{j=1}^n | l^j|^2}{| l^i|^2},
\end{equation}
and let $h_2$ be the Hermitian metric on $[E]$ over $X \setminus E$ determined by
$|s |_{h_2}^2 =1.$
Now define the Hermitian metric $h$ by
$h= \rho_1 h_1 + \rho_2 h_2,$
where $\rho_1$, $\rho_2$ is a partition of unity for the cover $(\pi^{-1}(D), X\setminus \pi^{-1}(D_{1/2}))$ of $X$, so that $h=h_1$ on $\pi^{-1}(D_{1/2})$.  The function $|s|_h^2$ on $X$ is given on $\pi^{-1}(D_{1/2})$ by
\begin{equation} \label{r}
|s|_h^2 (x) = |z^1|^2 + \cdots + |z^n|^2=:r^2,\quad \textrm{for } \pi(x) = (z^1, \ldots, z^n). 
\end{equation}
On $\pi^{-1} (D_{1/2} \setminus \{0 \})$, the curvature $R(h)$ of $h$ is given by
\begin{equation} \label{Rh}
R(h) = - \frac{\sqrt{-1}}{2\pi} \partial \ov{\partial} \log ( |z^1|^2 + \cdots + |z^n|^2).
\end{equation}
We have the following lemma (for a proof, see \cite{GH}, p. 187).

\begin{lemma} \label{lemmaomegaX}
For sufficiently small $\varepsilon_0>0$,
\begin{equation} \label{omegaXdefn}
\omega_X = \pi^* \omega_Y - \varepsilon_0 R(h)
\end{equation}
is a K\"ahler form on $X$.
\end{lemma}

From now on we fix $\varepsilon_0>0$ as in the lemma, with $\omega_X$ defined by (\ref{omegaXdefn}).  One can see from (\ref{Rh}) that in $\pi^{-1} (D_{1/2} \setminus \{0\})$,  which we can identify with $D_{1/2} \setminus \{ 0 \}$, the metric $\omega_X$ has the form:
\begin{equation}\label{omegaX}
\omega_X = \pi^* \omega_Y +  \frac{\sqrt{-1}}{2\pi} \frac{\varepsilon_0}{r^2} \sum_{i,j=1}^n \left(\delta_{ij} - \frac{\ov{z^i} z^j}{r^2} \right) dz^i \wedge d\ov{z^j},
\end{equation}
for $r$ given by (\ref{r}).
It is then not difficult to prove:

\begin{lemma} \label{lemmacomp}
There exist positive  constants $C$, $C'$ such that
\begin{equation} \label{ineq1}
\pi^* \omega_Y \le \omega_X \le C \frac{\pi^* \omega_Y}{|s|^2_h}
\end{equation}
and
\begin{equation} \label{ineq2}
\frac{1}{C'} \pi^* \omega_Y \le \omega_0 \le C' \frac{\pi^* \omega_Y}{|s|^2_h}.
\end{equation}
\end{lemma}
\begin{proof} Note that (\ref{ineq2}) follows immediately from (\ref{ineq1}) since $\omega_X$ and $\omega_0$ are uniformly equivalent K\"ahler metrics on $X$.
 It suffices to prove (\ref{ineq1}) in $\pi^{-1}(D_{1/2} \setminus \{0\})$.  We  make use of (\ref{omegaX}).
The first inequality of (\ref{ineq1}) follows from the fact that 
 if $Y^i$ is any $T^{1,0}$ vector then by the Cauchy-Schwarz inequality,
\begin{equation}
\sum_{i,j} \frac{\ov{z^i} z^j }{r^2} Y^i \ov{Y^j} = \frac{1}{r^2} \sum_i(\ov{z^i} Y^i) \sum_j (z^j \ov{Y^j}) \le  |Y|^2 = \sum_{i,j}\delta_{ij} Y^i \ov{Y^j}.
\end{equation}
The second inequality of (\ref{ineq1}) follows from the fact that $\ov{z^i} z^j$ is positive semi-definite. \qed

\end{proof}

The following lemma contains the key estimates of this section.

\begin{lemma} \label{lemmaest1}
There exists $\delta>0$ and a uniform constant $C$ such that for $\omega= \omega(t)$ a solution of the K\"ahler-Ricci flow:
\begin{enumerate}
\item[(i)] $\displaystyle{
 \omega \le \frac{C}{|s|_{h}^2}{ \pi^* \omega_Y}}$.
\item[(ii)] $\displaystyle{
 \omega \le \frac{C}{|s|_{h}^{2(1- \delta)}} \,\omega_0}$.
\end{enumerate}
\end{lemma}
\begin{proof}

Fix $0 < \ve \le 1$.  We will apply the maximum principle to the quantity
\begin{equation}
Q_{\ve} = \log \tr{\omega_0}{\omega} + A \log \left( |s|_{h}^{2+2\ve} \tr{\pi^* \omega_Y}{\omega} \right) - A^2 \varphi,
\end{equation}
on $X\setminus E$ where $A$ is a constant to be determined.
From Lemma \ref{lemmacomp} we have
\begin{equation} \label{oYX}
|s|^2_{h}\tr{\pi^*\omega_Y}{\omega} \le C\tr{\omega_0}{\omega}.
\end{equation}
It follows that at any fixed time $t$, $Q_{\ve}(x,t)$ tends to negative infinity as $x \in X$ tends to $E$.  Suppose there exists $(x_0, t_0) \in X\setminus E \times (0,T)$ with $\sup_{X\setminus E \times [0,t_0]} Q_{\ve} = Q_{\ve}(x_0,t_0)$. At $(x_0,t_0)$ we have
\begin{eqnarray}  \nonumber
0 \le \left( \ddt{} - \Delta \right) Q_{\ve} & \le & C \tr{\omega}{\omega_0} - A \tr{\omega} ( A \hat{\omega}_{t_0} - (1 +\varepsilon) R(h)- C' \pi^* \omega_Y) \\ && \mbox{} - A^2 \log \frac{\omega^n}{\Omega} + A^2n, \label{Hve1}
\end{eqnarray}
for some uniform constants $C, C'$.  Indeed, to see (\ref{Hve1}), we note that for any fixed K\"ahler metric $\tilde{\omega}$ on $X$, if $\omega=\omega(t)$ solves the K\"ahler-Ricci flow, we have by a well-known computation (\cite{Y1, A, Cao}).
\begin{eqnarray} \nonumber
\left( \ddt{} - \Delta \right) \log \tr{\tilde{\omega}}{\omega} & = & \frac{1}{\tr{\tilde{\omega}}{\omega}  } \left( - g^{i \ov{j}} \tilde{R}_{i \ov{j}}^{ \ \ k \ov{\ell}} g_{k \ov{\ell}} - g^{i \ov{j}} \tilde{g}^{k \ov{\ell}} g^{p\ov{q}} \tilde{\nabla}_i g_{k \ov{q}} \tilde{\nabla}_{\ov{j}} g_{p\ov{\ell}} + \frac{ |\nabla \tr{\tilde{\omega}}{\omega}|^2}{\tr{\tilde{\omega}}{\omega}} \right) \\
& \le & \tilde{C} \, \tr{\omega}{\tilde{\omega}}, \label{tr1}
\end{eqnarray}
for $\tilde{C}$ depending only on the lower bound of the bisectional curvature of $\tilde{g}$.  Applying (\ref{tr1}) with $\tilde{\omega} = \omega_0$ and then $\tilde{\omega}= \pi^*\omega_Y$ at the point $(x_0, t_0)$ we obtain (\ref{Hve1}), where in the second case the constant $\tilde{C}$ in (\ref{tr1}) depends only on the lower bound of the bisectional curvature of $\omega_Y$ on $Y$.

  But by Lemma \ref{lemmaomegaX} and the definition of $\hat{\omega}_t$ we can choose $A$ sufficiently large (and independent of $\ve$, $t_0$) so that
\begin{equation}
A \left( A \hat{\omega}_{t_0} - (1 +\varepsilon) R(h)- C' \pi^* \omega_Y \right) \ge (C+1)\omega_0.
\end{equation}
Then at $(x_0, t_0)$ we have
\begin{equation}
\left( A^2 \log  \frac{\omega^n}{\Omega} +  \tr{\omega}{\omega_0} \right)(x_0,t_0) \le C.
\end{equation}
But, for $c>0$, since the map $y \mapsto \log y + c/y$ is uniformly bounded from below and tends to infinity as $y \rightarrow 0^+$, this  implies that
\begin{equation}
(\tr{\omega}{\omega_0}) (x_0,t_0) \le C.
\end{equation}
Since the volume form $\omega^n$ is uniformly bounded from above by Lemma \ref{Linfinity}, we see that
\begin{equation*} \label{trdet}
(\tr{\omega_0}{\omega})(x_0,t_0) \le \frac{1}{(n-1)!}   \left(\tr{\omega}{\omega_0}\right)^{n-1} (x_0,t_0) \, \left(\frac{\omega^n}{\omega_0^n}\right) (x_0,t_0) \le C.
\end{equation*}
Moreover, from (\ref{oYX}) we see that
\begin{equation}
(|s|_h^2 \tr{\pi^*\omega_Y}{\omega})(x_0,t_0) \le C,
\end{equation}
and by Lemma \ref{Linfinity}, $\varphi$ is uniformly bounded.  It follows that 
\begin{equation} \label{He}
Q_{\ve} \le C,
\end{equation}
for $C$ independent of $\varepsilon$.   Letting $\ve \rightarrow 0$ and applying again (\ref{oYX})  immediately gives (i).
Using the fact that $\tr{\omega_0}{\omega} \le C \tr{\pi^*\omega_Y} \omega$ (by
 Lemma \ref{lemmacomp}), we also have
\begin{equation}
\log \left(  |s|_{h}^{2A} (\tr{\omega_0}{\omega})^{A+1} \right) \le C,
\end{equation}
and hence for some $\delta>0$, 
\begin{equation}
\tr{\omega_0}{\omega} \le \frac{C}{|s|^{2(1-\delta)}_{h}},
\end{equation}
giving
 (ii). \qed
\end{proof}

The estimate (ii) of Lemma \ref{lemmaest1} shows that, for any tangent vector $v$ on $X$ of unit length with respect to $\omega_0$, the quantity $|v|^2_{\omega}$ is bounded from above by $C/|s|_h^{2(1-\delta)}$, where $C$ may depend on $v$.  Although it is not strictly needed for the proof of the main theorem,  we can improve the exponent $2(1-\delta)$ to 1 if the vector $v$ is in a `radial direction' pointing away from the exceptional divisor.  More precisely, consider the holomorphic vector field
$$z^i \frac{\partial}{\partial z^i},$$
defined on the unit ball $D$.  This  defines via $\pi$ a holomorphic vector field $V$ on $\pi^{-1}(D) \subset X$ which vanishes to order 1 along the exceptional divisor $E$.  We then extend $V$ to be a smooth $T^{1,0}$ vector  field on the whole of $X$.   We then have the following lemma.

\begin{lemma} \label{lemmaradial} For $\omega=\omega(t)$ a solution of the K\"ahler-Ricci flow, we have the estimate
\begin{equation} \label{estimateV}
|V|_{\omega}^2 \le C|s|_h.
\end{equation}
for a uniform $C$.  Locally, in $D_{1/2} \setminus \{0\}$ we have
\begin{equation} \label{estimateW}
|W|^2_{g} \le \frac{C}{r},
\end{equation}
for $W = \sum_{i=1}^n \left( \frac{x^i}{r} \frac{\partial}{\partial x^i} + \frac{y^i}{r} \frac{\partial}{\partial y^i} \right)$ the unit length radial vector field with respect to $g_{\textrm{Eucl}}$, where $z^i = x^i + \sqrt{-1} y^i$.
\end{lemma}

In the statement and proof of the lemma, we are identifying $\pi^{-1}(D_{1/2} \setminus \{0 \})$ with $D_{1/2} \setminus \{0\}$ via the map $\pi$, and writing $g$ for the K\"ahler metric $(\pi^{-1})^* g$  `downstairs' on the punctured ball $D_{1/2} \setminus \{ 0 \} \subset Y$.  We remark that a related result about holomorphic vector fields and the behavior of the K\"ahler-Ricci flow can be found in \cite{PSSW3}.

\begin{proof}[Proof of Lemma \ref{lemmaradial}]
From the expression (\ref{omegaX}) we have, in $D_{1/2}$,
\begin{equation}
|V|^2_{\omega_X} = |V|^2_{\pi^* \omega_Y}.
\end{equation}
It follows that $|V|^2_{\omega_0}$ is uniformly equivalent to $|s|^2_h=r^2$ in $D_{1/2}$.  Hence there exists a positive constant $C(t)$ depending on $t$ such that
\begin{equation} \label{V2}
\frac{1}{C(t)} |s|_h^2 \le |V|_{\omega}^2 \le C(t) |s|^2_h.
\end{equation}
Now define a Hermitian metric $\tilde{\omega}_Y$ on $Y$ by
\begin{equation}
\tilde{\omega}_Y = \oeuc, \quad \textrm{on } D,
\end{equation}
and extending in an arbitrary way to be a smooth Hermitian metric on $Y$.
For small $\ve>0$, we consider the quantity
\begin{equation}
Q_{\ve} = \log \left( |V|_{\omega}^{2+2\ve} \tr{\pi^*\tilde{\omega}_Y}{\omega} \right) -t.
\end{equation}
Since $\tilde{\omega}_Y$ is uniformly equivalent to $\omega_Y$, we see that for fixed $t$, using (\ref{oYX}) and (\ref{V2}),  $(|V|^{2+2\ve}_{\omega} \tr{\pi^*{\tilde{\omega}_Y}}{\omega})(x)$ tends to zero as $x$ tends to $E$ and thus $Q_{\ve}(x)$ tends to negative infinity.  We now apply the maximum principle to $Q_{\ve}$. Since $Q_{\ve}$ is uniformly bounded from above on the complement of $D_{1/2} \setminus \{0 \}$, we only need to rule out the case when $Q_{\ve}$ attains its maximum at a point in $D_{1/2}\setminus \{0 \}$.
 Assume  that at some point $(x_0, t_0) \in D_{1/2} \setminus \{0 \} \times (0,T)$, we have    $\sup_{(X\setminus E) \times [0,t_0]}   Q_{\ve} = Q_{\ve}(x_0,t_0)$.
By (\ref{tr1}), we have in $D_{1/2} \setminus \{ 0\}$,
\begin{equation}
 \left( \ddt{} - \Delta \right) \tr{\pi^* \tilde{\omega}_Y}{\omega} \le 0,
 \end{equation}
since the curvature of $\tilde{\omega}_Y$ vanishes.
Next, compute in a normal coordinate system for $g$ at $(x_0,t_0)$,
\begin{eqnarray} 
(\ddt{} - \Delta) \log |V|^2_{\omega} & = & \frac{1}{|V|_{\omega}^2} \left( - g^{i \ov{j}} g_{k \ov{\ell}} (\partial_i V^k) ( \ov{\partial_j V^\ell})  + \frac{ | \nabla |V|^2_{\omega} |_{\omega}^2}{|V|_{\omega}^2} \right)  \le 0, \label{logV}
\end{eqnarray}
since, applying the Cauchy-Schwarz inequality,
\begin{eqnarray} \nonumber
 | \nabla |V|^2_{\omega} |_{\omega}^2 & = & g^{i \ov{j}} \partial_i (g_{k \ov{\ell}} V^k \ov{V^\ell}) \partial_{\ov{j}} (g_{p \ov{q}} V^p \ov{V^q}) \\ \nonumber
 &  = & g^{i \ov{j}} g_{k \ov{\ell}} g_{p \ov{q}} (\partial_i V^k) ( \ov{\partial_j V^q}) V^p \ov{V^\ell} \\ \nonumber
 & = & \sum_{i} \left( \sum_k (\partial_i V^k) \ov{V^k} \right) \left( \sum_p ( \ov{\partial_i V^p})V^p \right) \\
 & \le & \sum_i | \partial_i V|_{\omega}^2 |V|_{\omega}^2 = |V|_{\omega}^2 g^{i \ov{j}} g_{k \ov{\ell}} (\partial_i V^k) ( \ov{\partial_j V^\ell}).
\end{eqnarray}
Then  we obtain, at $(x_0, t_0)$,
\begin{eqnarray}
0 \le \left( \ddt{} - \Delta \right) Q_{\ve} \le - 1,
\end{eqnarray}
a contradiction.   Thus $Q_{\ve}$ is uniformly bounded from above.  Letting $\ve$ tend to zero we obtain
\begin{equation}
(\tr{\pi^*\omega_Y} \omega) |V|^2_{\omega} \le C,
\end{equation}
for some uniform constant $C$.   By Lemma \ref{lemmacomp} we have
$(\tr{\omega_0}{\omega}) |V|^2_{\omega} \le C,$
and since $|V|^2_{\omega} \le  (\tr{\omega_0}{\omega}) |V|^2_{\omega_0}$ this gives
\begin{equation}
|V|^4_{\omega} \le C |V|^2_{\omega_0}
\end{equation}
and (\ref{estimateV}) follows from the fact that $|V|^2_{\omega_0}$ is uniformly equivalent to $|s|^2_h$ in $D_{1/2}$.

Finally observe that 
$\textrm{Re} \left( \frac{1}{r} V \right) = \frac{1}{2} W$  in $D_{1/2} \setminus \{0\}$ and hence (\ref{estimateW}) follows from (\ref{estimateV}).
 \qed
\end{proof}

We end this  section with estimates on the lengths of spherical and radial paths in the punctured ball $D_{1/2} \setminus \{0\}$ which again we identify with its preimage under $\pi$.  We obtain in this way a diameter bound of $X$ with respect to $g(t)$.

\pagebreak[2]
\begin{lemma} \label{lemmakey}  We have
\begin{enumerate}
\item[(i)] For $0<r<1/2$, the diameter of the $2n-1$ sphere $S_r$ of radius $r$ in $D$ centered at the origin with the metric induced from $\omega$ is uniformly bounded from above, independent of $r$.
\item[(ii)]  For any $x \in D_{1/2} \setminus \{0 \}$, the length of a radial path $\gamma(\lambda) = \lambda x$ for $\lambda \in (0,1]$ with respect to $\omega$ is uniformly bounded from above by a uniform constant multiple of $|x|^{1/2}$.
\end{enumerate}
Hence the diameter of $D_{1/2} \setminus \{0 \}$ with respect to $\omega$ is uniformly bounded from above and
$$\emph{diam}_{\omega} X \le C.$$
\end{lemma}
\begin{proof}
For (i), we use  Lemma \ref{lemmaest1}.
Since $\omega_Y$ is equivalent to the Euclidean metric $\oeuc$ in these coordinates, we have
\begin{equation}
\omega \le \frac{C}{r^2} \oeuc,
\end{equation}
for $r^2= |z^1|^2 + \cdots + |z^n|^2$.  Write $\iota_r$ for the inclusion map $\iota_r: S_r \rightarrow D$.
 Since $d_{\iota_r^*\geuc}(p,q) \le \pi r$ for $p, q \in S_r$, we have, for any $0 < r < 1/2$,
\begin{equation} \label{sphere}
d_{\iota_r^*g}(p,q) \le \sqrt{C} r^{-1} d_{\iota_r^*\geuc}(p,q)  \le \sqrt{C} \pi, \quad \textrm{for } p,q \in S_r,
\end{equation}
and this gives (i).

For (ii), we use  the estimate (\ref{estimateW}) of Lemma \ref{lemmaradial}.  The
 length of the radial path $\gamma(\lambda) = \lambda x$, for $\lambda \in (0,1]$ is given by
\begin{equation}
\int_0^1 \sqrt{g_{\gamma(\lambda)}(\gamma'(\lambda), \gamma'(\lambda))} d\lambda \le C \sqrt{| x|} \int_0^1 \frac{1}{ \sqrt{\lambda}}  d\lambda \le C' |x|^{1/2}.
\end{equation}
We remark as an aside that one could also prove a weaker version of (ii) (with an upper bound $C|x|^{\delta}$ for $\delta>0$) using Lemma \ref{lemmaest1}.

From (i) and (ii) we immediately obtain an upper bound for the diameter of $D_{1/2} \setminus \{0 \}$  with respect to $\omega$.  Combined with Lemma \ref{lemmack}, this is enough to obtain the diameter bound of $X$.  Indeed it only remains to check the case when $p$ and $q$ are points in $\pi^{-1}(D_{1/2})$, with one or both possibly lying on the exceptional divisor $E$.  Take any sequences of points $p_i$ and $q_i$ in $X$, such that  $p_i \rightarrow p$ and $q_i \rightarrow q$ with respect to any fixed metric, and so that $p_i, q_i$ correspond to points in the punctured ball $D_{1/2} \setminus \{0 \}$.  Since $d_{\omega}(p_i, q_i)$ is uniformly bounded from above, independent of $i$, we can take the limit as $i$ tends to infinity to obtain an upper bound for $d_{\omega} (p,q)$.\qed
\end{proof}

\begin{remark}  As $\lambda \rightarrow 0^+$ the path $\gamma(\lambda)$ in Lemma \ref{lemmakey}, regarded as a path in $X$, tends to a point $x_0 \in E$ determined by $x \in D_{1/2} \setminus \{0 \}$.  Indeed $x$ is contained in a complex line in $\mathbb{C}^n$, to which corresponds a point  $l=[l^1, \ldots, l^n] \in \mathbb{P}^{n-1}$.  The point $x_0\in E$ is the point $(0,l) \in \tilde{D}$, with the notation of (\ref{Dtilde}).
\end{remark}

\section{Gromov-Hausdorff convergence} \label{sectionGH}

We now prove parts (ii) and (iii) in the definition of canonical surgical contraction, under the assumptions of Theorem \ref{thmblowup} (again supposing for simplicity that we only have one exceptional divisor $E$).  We define a compact metric space $(Y, d_T)$ and show that $(X,g(t))$ converges in the Gromov-Hausdorff sense to this metric space.

From Lemma \ref{lemmack} we have established part (i) in the definition of canonical surgical contraction, that $\omega(t)$ converges smoothly on compact subsets of $X \setminus E$ to a $\omega_T$, a K\"ahler metric on $X\setminus E$.
Recall that  $X \setminus E$ can be identified with $Y \setminus \{ y_0 \}$ via the map $\pi$.  Abusing notation, we will write $\omega_T$ for the smooth K\"ahler form $(\pi|_{X\setminus E}^{-1})^* \omega_T$ on $Y \setminus \{ y_0 \}$, and $g_T$ for the corresponding K\"ahler metric.  We extend  $g_T$ to a nonnegative (1,1)-tensor  $\tilde{g}_T$ on the whole space $Y$ by setting $\tilde{g}_T|_{y_0}( \cdot, \cdot) =0$.  Of course, $\tilde{g}_T$ may be discontinuous at $y_0$.

We use $\tilde{g}_T$ to define a distance function $d_T$ on $Y$ as follows.

\begin{definition}  \label{defndy}   Define a distance function $d_T: Y \times Y \rightarrow \mathbb{R}$ by
\begin{equation}
d_T(y_1, y_2) = \inf_{\gamma}  \int_{0}^1  \sqrt{ \tilde{g}_T ( \gamma'(s), \gamma'(s))}  ds,
\end{equation}
where the infimum is taken over all piecewise smooth paths $\gamma: [0,1] \rightarrow Y$ with $\gamma(0)=y_1$, $\gamma(1)=y_2$.
\end{definition}

We make a couple of remarks about this definition.  First, it is easily seen that   $d_T$ is well-defined: the quantity $d_T(y_1, y_2)$ is finite for all $y_1, y_2 \in Y$.  Indeed, we only need to check the case when $y_1 =y_0$ and $y_2 \neq y_0$.  From Lemma \ref{lemmakey}, the lengths of radial paths in $D_{1/2}$ emanating from $y_0 \in Y$, with respect to $\omega(t)$ are uniformly bounded, independent of $t$.  It follows that there is a curve of finite length with respect to $\tilde{g}_T$ from $y_1$ to a fixed point $y' \in D_{1/2} \setminus \{0 \}$, and hence $d_T(y_1, y_2)$ is finite.
Second, note that $d_T$ is independent of the extension of $g_T$ to $Y$.

We have the following lemma.

\begin{lemma}  $(Y, d_T)$ is a compact metric space.
\end{lemma}

\begin{proof}  First we show that $d_T$ defines a metric on $Y$.  The only somewhat non-trivial statement we need to check is the implication
\begin{equation} \label{im}
y_1 \neq y_2 \qquad \Longrightarrow \qquad  d_T (y_1, y_2) >0.
\end{equation}
To see (\ref{im}), we may assume without loss of generality that $y_1 \neq y_0$.  Then there is a small coordinate ball $B$ centered at $y_1$ with $B \subset Y \setminus \{ y_0 \}$ and $y_2 \notin B$.  Since $g_T$ is a smooth K\"ahler metric on $B$, the length of any path from $y_1$ to the boundary of $B$ is bounded below by a strictly positive constant.  It follows that $d_T(y_1, y_2)>0$.

Now we will show that $(Y, d_T)$ is compact. It suffices to show that every sequence of points $\{y_j\}$ in $Y$ has a convergent subsequence with respect to $d_T$.
Let $d_{\omega_Y}$ be the distance function on $Y$ associated to  the smooth Kahler form $\omega_Y$ on $Y$.  Since $(Y,d_{\omega_Y})$ is compact, there is a  subsequence $\{ y_{j_k} \}$ of $\{ y_j \}$ which converges in the $d_{\omega_Y}$ metric to a point $y_{\infty}$ in $Y$.  We will show that $\{ y_{j_k} \}$ converges to $y_{\infty}$ in the $d_T$ metric.

Suppose $y_{\infty} = y_0$.  Then for $k$ sufficiently large, $\{ y_{j_k} \}$ is contained in the ball $D \subset Y$.  The length of the radial path from $y_0$ to $y_{j_k}$ with respect to $d_T$ tends to zero as $k$ tends to infinity, by Lemma \ref{lemmakey}.  Hence $d_T (y_{j_k}, y_0 ) \rightarrow 0$ as $k \rightarrow \infty$.

Now suppose $y_\infty \neq y_0$.  For $k$ sufficiently large, the points $y_{j_k}$ lie in an open geodesic ball $B_{\varepsilon}(y_{\infty})$ with respect to $g_Y$, centered at $y_{\infty}$ of radius $\varepsilon>0$,    with $\ov{B_{\varepsilon}(y_{\infty})} \subset Y \setminus \{ y_0\}$.   But in this geodesic ball, the metric $g_T$ is uniformly equivalent to the metric $g_Y$.   Since $d_{\omega_Y} (y_{j_k}, y_{\infty}) \rightarrow 0$ it follows that $d_T (y_{j_k}, y_{\infty}) \rightarrow 0$. \qed
\end{proof}

Now let $d_{\omega}= d_{\omega(t)}$ be the distance function on $X$ associated to the evolving K\"ahler metric $\omega$.  We prove the following:

\begin{lemma} \label{diameterE} There exists a uniform constant $C$ such that for any points $p, q \in E$, and any $t \in [0,T)$,
\begin{equation} \label{dpq}
d_{\omega}(p,q) \le C(T-t)^{1/3}.
\end{equation}
\end{lemma}

\begin{proof}   We first deal with the case of complex dimension 2.  In this case we can identify $E$ with $\mathbb{P}^1$ and
\begin{equation} \label{n2}
\int_E \omega(t) =  \int_E \frac{1}{T} \left((T-t) \omega_0 + t \pi^* \omega_Y \right) =\frac{T-t}{T} \int_E  \hat{\omega}_0 \le C(T-t).
\end{equation}
Write $\varepsilon = (T-t)^{1/3}>0$, which we may assume is sufficiently small.
  Without loss of generality, we may assume that the two points $p$ and $q$ in $E \cong \mathbb{P}^1$ lie in a fixed coordinate chart $U$ whose image under the holomorphic coordinate $z=x+\sqrt{-1}y$ is  a ball of radius $2$ in $\mathbb{C}= \mathbb{R}^2$ with respect to the Euclidean metric.  Moreover, we may assume that $p$ is represented by the origin in $\mathbb{C}= \mathbb{R}^2$, and $q$ is represented by the point $(x_0,0)$ with $0<x_0<1$ and that the rectangle
\begin{equation}
\mathcal{R} = \{ (x,y) \in \mathbb{R}^2 \ | \ 0 \le x \le x_0, \ -\varepsilon \le y \le \ve \} \subset \mathbb{R}^2 = \mathbb{C},
\end{equation}
is contained in the image of $U$.  Now in $\mathcal{R}$, the fixed metric $g_0$ induced from the metric $g_0$ on $X$ is uniformly equivalent to the Euclidean metric.  Thus from (\ref{n2}),
\begin{equation}
\int_{-\ve}^{\ve} \left( \int_{0}^{x_0} (\tr{\hat{g}_0}{g}) dx \right) dy  = \int_{\mathcal{R}} (\tr{\hat{g}_0}{g}) dxdy \le  C(T-t).
\end{equation}
Hence there exists $y' \in (-\ve, \ve)$ such that
\begin{equation}
\int_0^{x_0} (\tr{\hat{g}_0}g)(x, y') dx \le \frac{C}{\ve} (T-t) = C (T-t)^{2/3}.
\end{equation}
Now let $p'$ and $q'$ be the points represented by coordinates $(0,y')$ and $(x_0, y')$.  Then, considering the horizontal path $s \mapsto (s, y')$  between $p'$ and $q'$, we have
\begin{eqnarray} \nonumber
d_{\omega}(p',q') & \le & \int_0^{x_0} \left(\sqrt{ g( \partial_x, \partial_x)} \right)(x,y') dx \\ \nonumber 
& = & \int_0^{x_0} \left(\sqrt{ \tr{\hat{g}_0}{g}} \sqrt{ \hat{g}_0( \partial_x, \partial_x)} \right)(x,y') dx\\ \nonumber
& \le & \left( \int_0^{x_0} (\tr{\hat{g}_0}{g})(x,y') dx\right)^{1/2} \left( \int_0^{x_0} \left( \hat{g}_0 ( \partial_x, \partial_x) \right)(x,y')dx \right)^{1/2} \\
& \le & C(T-t)^{1/3}.
\end{eqnarray}
We now claim that
\begin{equation} \label{pq}
d_{\omega}(p,p') + d_{\omega}(q,q') \le C \varepsilon,
\end{equation}
for some constant $C$, and then this will give (\ref{dpq}).  To see (\ref{pq}) we will work in the open ball $D \subset Y$ rather than on $X$.  The points $p$ and $p'$ in $E= \mathbb{P}^1$ can be identified with complex lines $L_p$ and $L_{p'}$ in $\mathbb{C}^2$.  We will also write $L_p$ and $L_{p'}$ for the intersection of these lines with $D \subset \mathbb{C}^2$.  Note that, for example,  a path in $L_p \subset D$ emanating from the origin corresponds in $X$ to a path which starts at $p \in E$ and is transversal to $E$.

For any $0<r < 1/2$, to be chosen later, we can find points $\tilde{p}$ and  $\tilde{p}'$ on $S_r \cap L_p$ and $S_r \cap L_{p'}$ respectively, where $S_r$ is the sphere of radius $r$ centered at the origin, such that
\begin{equation}
d_{\omega} (\tilde{p}, \tilde{p}') \le C \varepsilon.
\end{equation}
This follows from the same reasoning as in the argument of (\ref{sphere}), together with the fact that the distance between $p$ and $p'$ in the standard metric on $\mathbb{P}^1$ is bounded by a constant multiple of $\varepsilon$.  On the other hand choosing $r$ sufficiently  small we can see that, by the estimate of $g$ in the radial direction, the distances from $p$ to $\tilde{p}$ and from $p'$ to $\tilde{p}'$ can be made arbitrarily small (and thus less than a multiple of $\varepsilon$).  Then
\begin{equation}
d_{\omega}(p,p') \le d_{\omega}(p, \tilde{p}) + d_{\omega}(\tilde{p}, \tilde{p}') + d_{\omega} (\tilde{p}', p') \le C \varepsilon.
\end{equation}
Thus (\ref{pq}) follows after applying the same argument to $q$.

For the case of complex dimension $n>2$, we identify $E$ with $\mathbb{P}^{n-1}$.   Given $p, q$ in $\mathbb{P}^{n-1}$ there exists a projective line $S \cong \mathbb{P}^1$ containing the points $p$ and $q$.  Then
\begin{equation} \label{S}
\int_S \omega(t) \le C(T-t),
\end{equation}
and we obtain the same result by a similar argument to the one given above. This completes the proof of the lemma. \qed
\end{proof}

\begin{lemma} \label{diam}  For any $\varepsilon>0$ there exists $\delta_0>0$ and $T_0 \in [0,T)$ such that
\begin{equation} \label{smallballdT}
\emph{diam}_{d_T}  D_{\delta_0} < \varepsilon.
\end{equation}
and
\begin{equation} \label{smallball1}
\emph{diam}_{{\omega(t)}} \pi^{-1}(D_{\delta_0})  < \varepsilon, \quad \textrm{for all } \ t \in [T_0,T).
\end{equation}
\end{lemma}
\begin{proof} For (\ref{smallballdT}), let $p,q \in D_{\delta_0}$.  Applying Lemma \ref{lemmakey},  the lengths of radial paths $\gamma_p, \gamma_q: [0,1] \rightarrow D_{\delta_0}$  from the origin to $p$, $q$ respectively with respect to $\omega(t)$ are each less than $\ve/4$ say after choosing  $\delta_0$ small enough.  Since $\omega(t)$ converges to $\omega_T$ smoothly on compact subsets of $D_{\delta_0} \setminus \{ 0 \}$, we obtain $d_T(p,q)  \le \ve/2 < \ve$.

For (\ref{smallball1}), combine Lemma \ref{lemmakey} with Lemma \ref{diameterE}.
\qed
\end{proof}

Then, using in addition the first part of Lemma \ref{lemmack}, we have established part (ii) in the definition of canonical surgical contraction.

\begin{remark}  In fact, we can prove the following statement:  for any $\varepsilon>0$ there exists $\delta_0>0$ and $T_0 \in [0,T)$ such that
\begin{equation} \label{smallball2}
\emph{diam}_{\omega(t)} \pi^{-1}(D_{\delta_0} \setminus \{0 \}) < \varepsilon, \quad \textrm{for all } \ t \in [T_0,T).
\end{equation}
The proof of this statement uses the ideas of Lemma \ref{diameterE} together with an application of Stoke's Theorem.  If we could replace the conclusion (\ref{smallball2}) by the assertion
\begin{equation}
\emph{diam}_{\omega_T} \pi^{-1}(D_{\delta_0} \setminus \{0 \}) < \varepsilon,
\end{equation}
then it would follow that $(Y, d_T)$ can be identified with the metric completion of $(X \setminus E, g_T)$  (see Remark \ref{remark1}).
\end{remark}

We now show that $(X,g(t))$ converges to the compact metric space $(Y, d_T)$ in the sense of Gromov-Hausdorff, as $t \rightarrow T^-$.
We recall the definition of Gromov-Hausdorff distance $d_{\textrm{GH}}(X,Y)$ between two metric spaces $(X, d_X)$ and $(Y,d_Y)$, using the  characterization  given in \cite{F}.   Define  $d_{\textrm{GH}}(X,Y)$ to be the infimum of all $\varepsilon>0$ such that the following holds.   There exist maps $F: X \rightarrow Y$ and $G: Y \rightarrow X$   such that
\begin{equation} \label{GH1}
|d_X(x_1, x_2) - d_Y (F(x_1), F(x_2))|\leq \varepsilon, \quad \textrm{for all } x_1, x_2 \in X
\end{equation} and
\begin{equation} \label{GH2}
d_X(x, G\circ F(x)) \le \varepsilon, \quad \textrm{for all } x \in X
\end{equation}
and the two symmetric properties for $Y$ also hold.   We do not require the maps $F$ and $G$ to be continuous.

Fix $\varepsilon>0$.  We will show that for $t$ sufficiently close to $T$  the Gromov-Hausdorff distance between $(X, d_{\omega(t)})$ and $(Y, d_Y)$ is less than $\varepsilon$.
Let $F: X \rightarrow Y$ be the blow-down map $\pi$ and let $G: Y\rightarrow X$ be a map with $G=(\pi|_{X\setminus E})^{-1}$ on $Y\setminus \{y_0 \}$ and $G(y_0)$ chosen to be any point on $E$.

First observe that, by definition,
\begin{equation} \label{ghyy}
d_T (y, F \circ G(y))=0, \quad \textrm{for all } y \in Y,
\end{equation}
and
\begin{equation} \label{ghxx1}
d_{\omega(t)} (x, G \circ F(x))=0, \quad \textrm{for all } x \in X \setminus E.
\end{equation}
Moreover, from Lemma \ref{diameterE} we have, for $t$ sufficiently close to $T$,
\begin{equation} \label{ghxx2}
d_{\omega(t)} (x, G \circ F(x)) < \varepsilon, \quad \textrm{for all } x \in  E.
\end{equation}

We will now show that for $t$ sufficiently close to $T$,
\begin{equation} \label{g1}
| d_{\omega(t)} (x_1, x_2) - d_T(F(x_1), F(x_2)) | \le \varepsilon, \quad \textrm{for all } x_1, x_2 \in X.
\end{equation}
By Lemma \ref{diam}, there exist $\delta_0>0$ and $T_0\in [0,T)$ such that for all $x_1, x_2 \in \pi^{-1}(\overline{D_{\delta_0}})$ and all $t \in [T_0,T)$ we have
\begin{equation} \label{g2}
 d_{\omega(t)} (x_1, x_2) +  d_T(F(x_1), F(x_2))  \le \frac{\varepsilon}{2}.
\end{equation}
Hence we may assume without loss of generality that $x_1, x_2$ lie in $X \setminus \pi^{-1}(D_{\delta_0})$.  From the definition of $d_T$, there exists a smooth path $\gamma: [0,1] \rightarrow Y$ with $\gamma(0)= F(x_1)$ and $\gamma(1)= F(x_2)$, such that
\begin{equation}
\left| d_T (F(x_1), F(x_2)) - \int_0^1 \sqrt{ \tilde{g}_T (\gamma'(s), \gamma'(s))} ds \right| < \frac{\varepsilon}{4}.
\end{equation}
Let $s_1$ be the smallest value in $[0,1]$ such that $\gamma(s_1)$ lies in $\overline{D_{\delta_0}}$, and let $s_2$ be the largest value in $[0,1]$ such that $\gamma(s_2)$ lies in $\overline{D_{\delta_0}}$.  Then we have
\begin{eqnarray} \nonumber
d_{\omega(t)} (x_1, x_2) & \le & \int_0^{s_1} \sqrt{g(t)( \gamma'(s), \gamma'(s))} ds + \int_{s_2}^1 \sqrt{g(t)(\gamma'(s), \gamma'(s))}ds \\ \label{g3}
&& \mbox{} + d_{\omega(t)} (F^{-1}(\gamma(s_1)), F^{-1}(\gamma(s_2))),
\end{eqnarray}
From (\ref{g2}), we have for $t\in [T_0,T)$,
\begin{equation} \label{g35}
 d_{\omega(t)} (F^{-1}(\gamma(s_1)), F^{-1}(\gamma(s_2))) < \frac{\varepsilon}{2}.
\end{equation}
Since $g(t)$ converges to $g_T$ uniformly on compact subsets of $X\setminus E$, we can choose $T_0$ sufficiently close to $T$ (depending on $\delta_0$) such that
\begin{equation} \label{g4}
\int_0^{s_1} \sqrt{g(t)( \gamma'(s), \gamma'(s))} ds + \int_{s_2}^1 \sqrt{g(t)(\gamma'(s), \gamma'(s))}ds \le \frac{\varepsilon}{4} + \int_0^1 \sqrt{ \tilde{g}_T (\gamma'(s), \gamma'(s))} ds.
\end{equation}
Combining (\ref{g3}) with (\ref{g35}), (\ref{g4}) and (\ref{g1}) we obtain
\begin{equation} \label{g6}
d_{\omega(t)}(x_1,x_2) - d_T(F(x_1), F(x_2)) \le \varepsilon.
\end{equation}

For the lower bound of $d_{\omega(t)}(x_1,x_2) - d_T(F(x_1), F(x_2))$, choose a path $\tilde{\gamma}: [0,1] \rightarrow X$, depending on $t$, with $\tilde{\gamma}(0)=x_1$ and $\tilde{\gamma}(1)=x_2$ such that
\begin{equation}
\left| d_{\omega(t)} (x_1, x_2) - \int_0^1 \sqrt{g(t)(\tilde{\gamma}'(s), \tilde{\gamma}'(s))} ds \right| < \frac{\varepsilon}{4}.
\end{equation}
We can then argue in a similar way as above (considering the first and last points of the path $\tilde{\gamma}$ that lie in $\pi^{-1}(\ov{D_{\delta}})$) to obtain for $t$ sufficiently close to $T$,
\begin{equation}
d_{\omega(t)}(x_1,x_2) - d_T(F(x_1), F(x_2)) \ge - \varepsilon.
\end{equation}
Thus for $t$ sufficiently close to $T$ we obtain for all $x_1, x_2$ in $X$,
\begin{equation} \label{ghx}
\left| d_{\omega(t)}(x_1,x_2) - d_T(F(x_1), F(x_2)) \right| \le \varepsilon.
\end{equation}

A similar argument shows that  for all $y_1, y_2 \in Y$,
\begin{equation} \label{ghy}
\left| d_T(y_1, y_2) - d_{\omega(t)}(G(y_1), G(y_2) ) \right| <\varepsilon,
\end{equation}
for $t$ sufficiently close to $T$.
Combining (\ref{ghyy}), (\ref{ghxx1}), (\ref{ghxx2}), (\ref{ghx}), (\ref{ghy}) completes the proof of the Gromov-Hausdorff convergence.  This establishes part (iii) in the definition of canonical surgical contraction under the assumptions of Theorem \ref{thmblowup}.

\pagebreak[3]
\section{Higher order estimates for $\omega(t)$ as $t \rightarrow T^-$} \label{sectionhigher}

Under the assumption of Theorem \ref{thmblowup}, we have 
seen that $(X, g(t))$ converges to $(Y,d_T)$ in the sense of Gromov-Hausdorff as $t\rightarrow T^-$. Furthermore, $\omega(t)$ converges in $C^\infty$ on any compact subset of $X\setminus \cup_{i=1}^k E_i$ to a smooth K\"ahler metric $\omega_T$ on $X\setminus \cup_{i=1}^k E_i$.   In particular, we already have $C^{\infty}$ a priori estimates for $\omega(t)$ away from the divisors $E_i$ as $t$ approaches $T$.   However, to understand how the K\"ahler-Ricci flow can be continued past the singular time we need more precise estimates.

We assume as before that there is only one exceptional divisor $E$.
From Lemmas \ref{lemmack}, \ref{lemmacomp}   and  \ref{lemmaest1}, there exist positive constants $\delta$ and $C$ such that
\begin{equation} \label{so1}
\frac{|s|^2_h}{C}  \omega_0 \le \omega(t) \leq \frac{C}{|s|_h^{2(1-\delta)}} \omega_0.
\end{equation}
For convenience, we assume without loss of generality that $|s|_h \le 1$ on $X$.

We begin with an analogue of Yau's third order estimate \cite{Y1} (see also \cite{C1,Cao,PSS}).  Following \cite{PSS}, we
define an endomorphism $H=H(t)$ of the tangent bundle by $H^i_{\, \ell} = g_0^{i \ov{j}} g_{\ell \ov{j}}$ and consider the quantity $S=S(t)$ given by
\begin{equation}
S=  |\nabla H \, H^{-1}|^2,
\end{equation}
where, here and henceforth, the covariant derivative $\nabla$ and the norm $|\, \cdot \, |$ are taken with respect to the evolving metric $g$.

\begin{proposition} \label{propS} There exist positive constants $\alpha$ and $C$ such that for $t\in [0, T)$,

\begin{equation}
S \leq \frac{C}{|s|_h^{2\alpha}}.
\end{equation}

\end{proposition}

\begin{proof} From the computation in \cite{PSS} (see their equation (2.51)),  we have
\begin{equation} \label{S1}
\left( \ddt{}-\Delta \right) S \leq -|\overline{\nabla}(\nabla H \, H^{-1})|^2 - |\nabla(\nabla H \,  H^{-1})|^2 + C S + C| \nabla \textrm{Rm}(g_0)|^2,
\end{equation}
where $\textrm{Rm}(g_0)$ is the Riemannian curvature tensor of $g_0$.  Using (\ref{so1}) we can estimate:
\begin{equation} \label{S2}
| \nabla \textrm{Rm}(g_0) |^2 \le C |s|_h^{-K} (S+1),
\end{equation}
for a positive constant $K$.  Moreover, we have:
\begin{equation} \label{S3}
|\nabla S|\leq S^{1/2} ( |\overline{\nabla}(\nabla H \, H^{-1})| + |\nabla(\nabla H \, H^{-1})|),
\end{equation}
and, 
\begin{equation} \label{S4}
| \nabla |s|^{4K}_h | \le C|s|^{3K}_h, \quad |\Delta |s|^{4K}_h | \le C |s|^{3K}_h,
\end{equation}
where, here and henceforth, we are choosing $K$ sufficiently large.
 Combining (\ref{S1}), (\ref{S2}), (\ref{S3}) and (\ref{S4}), we obtain
\begin{eqnarray} \nonumber
\left( \ddt{}-\Delta \right) (|s|_h^{4 K } S) & = & |s|^{4K}_h \left( \ddt{} - \Delta \right) S - 2 \textrm{Re} ( \nabla |s|^{4K}_h \cdot \ov{\nabla} S) - (\Delta |s|_h^{4K})S  \\ \nonumber
&\leq& - |s|_h^{4K}( |\overline{\nabla}(\nabla H \, H^{-1})|^2 + |\nabla(\nabla H \, H^{-1})|^2) \\ \nonumber
&& \mbox{}  + C |s|^{3K}_h S^{1/2} ( |\overline{\nabla}(\nabla H \, H^{-1})| + |\nabla(\nabla H \, H^{-1})|) + C |s|^{2K}_h S +C  \\ \label{S5}
&\leq&  C (1+ |s|_h^{2K} S).
\end{eqnarray}

On the other hand, $\tr{\omega_0}{\omega}$ satisfies (cf. (\ref{tr1}))
\begin{eqnarray} \nonumber
\left(\ddt{} - \Delta \right) \tr{\omega_0}{\omega}  & = & - g^{k \ov{\ell}} R_{k \ov{\ell}}^{\ \ \, i \ov{j}}(g_0) g_{i \ov{j}} - g_0^{i \ov{j}} g^{k \ov{\ell}} g^{p\ov{q}} \nabla^0_k g_{i \ov{q}} \nabla^0_{\ov{\ell}} g_{p\ov{j}} \\
\label{S6}
& \le &  C|s|_h^{-K} - \frac{1}{C}|s|_h^K S  - \frac{1}{2}g_0^{i \ov{j}} g^{k \ov{\ell}} g^{p\ov{q}} \nabla^0_k g_{i \ov{q}} \nabla^0_{\ov{\ell}} g_{p\ov{j}} ,
\end{eqnarray}
where we have used again (\ref{so1}).   Here $\nabla^0$ denotes the covariant derivative with respect to $g_0$.  Recall the well-known estimate (see \cite{Y1} for example),
\begin{equation} \label{S7}
| \nabla \tr{\omega_0}{\omega}|^2 \le (\tr{\omega_0}{\omega}) g_0^{i \ov{j}} g^{k \ov{\ell}} g^{p\ov{q}} \nabla^0_k g_{i \ov{q}} \nabla^0_{\ov{\ell}} g_{p\ov{j}}.
\end{equation}
Compute
\begin{eqnarray} \nonumber
\lefteqn{\left(\ddt{} - \Delta\right) ( |s|_h^{K} \tr{\omega_0}{\omega} ) } \\ \nonumber
&\leq& -\frac{1}{C} |s|_h^{2K}S + C - 2 \textrm{Re}( \nabla |s|_h^{K} \cdot  \overline{\nabla} \tr{\omega_0}{\omega})- \frac{1}{2}|s|_h^K g_0^{i \ov{j}} g^{k \ov{\ell}} g^{p\ov{q}} \nabla^0_k g_{i \ov{q}} \nabla^0_{\ov{\ell}} g_{p\ov{j}}\\ \label{S8}
& \le & -\frac{1}{C} |s|_h^{2K}S + C,
\end{eqnarray}
where we have used the estimate
\begin{equation}
| 2 \textrm{Re}( \nabla |s|_h^{K} \cdot  \overline{\nabla} \tr{\omega_0}{\omega}) | \le C+\frac{1}{C} | \nabla |s|_h^K|^2 |\nabla \tr{\omega_0}{\omega} |^2\le  C + \frac{1}{2}|s|_h^K g_0^{i \ov{j}} g^{k \ov{\ell}} g^{p\ov{q}} \nabla^0_k g_{i \ov{q}} \nabla^0_{\ov{\ell}} g_{p\ov{j}},
\end{equation}
which follows from (\ref{S7}) and (\ref{so1}).

If we let $Q= |s|_h^{4K} S + A |s|_h^{K} \tr{\omega_0}{\omega} -Bt,$ for constants $A$ and $B$ then from (\ref{S5}) and (\ref{S8}) we obtain
\begin{equation}
\left( \ddt{} - \Delta \right) Q  <0,
\end{equation}
by choosing $A$ and then $B$ sufficiently large.
Applying the maximum principle gives a uniform upper bound for $Q$ and the proposition follows.
\qed
\end{proof}

We now show that the curvature of $g$ and all its covariant derivatives have similar bounds.

\begin{proposition} \label{propRm} For each integer $m \ge 0$ there exist $C_m, \alpha_m>0$ such that  for $t\in [0, T)$,
\begin{equation} \label{boundRm}
|\nablaR^m \emph{Rm}(g)| \leq \frac{C_m}{|s|_h^{2\alpha_m}},
\end{equation}
where we use $\nablaR= \frac{1}{2}(\nabla + \overline{\nabla})$ to denote the covariant derivative of $g$ as a Riemannian metric.
\end{proposition}

\begin{proof}  From a computation of Hamilton's (Corollary 13.3 of \cite{H1}), we have
\begin{equation}  \label{evolveRm}
 \ddt{} |\nablaR^m \textrm{Rm}|^2 = \Delta |\nablaR^m \textrm{Rm}|^2 - 2|\nablaR^{m+1} \textrm{Rm}|^2 + \sum_{i+j=m} \nablaR^i \textrm{Rm} * \nablaR^j \textrm{Rm} * \nablaR^m \textrm{Rm},
 \end{equation}
 for $\textrm{Rm}=\textrm{Rm}(g)$ and
where $*$ is used to denote a linear combination of tensors formed by contraction with the metric $g$.  
We prove (\ref{boundRm}) by induction.  For the case $m=0$ we compute from (\ref{evolveRm}) that
\begin{equation} \label{1Rm1}
\left( \ddt{}- \Delta \right) |\textrm{Rm}| \leq  C |\textrm{Rm}|^2 + C.
\end{equation}
On the other hand we have from Proposition \ref{propS} and the inequality (\ref{S1}),
\begin{equation} \label{Ric2}
\left( \ddt{}- \Delta \right) S \leq  -|\overline{\nabla}(\nabla H \, H^{-1})|^2 - |\nabla(\nabla H \,  H^{-1})|^2 + C |s|^{-K}_h,
\end{equation}
where $K$ is a sufficiently large constant, which we take to be much larger than the constant $\alpha$ from Proposition \ref{propS}.
We compute the evolution of $Q=|s|_h^{4K} |\textrm{Rm}| + A|s|^{2K}_h S$ for a constant $A$  to be determined later.
\begin{eqnarray} \nonumber
 \left( \ddt{}- \Delta \right) Q
& = & |s|^{4K}_h \left( \ddt{}- \Delta \right) |\textrm{Rm}| + A |s|^{2K}_h \left( \ddt{}- \Delta \right) S \\ \nonumber
&&\mbox{} - 2\textrm{Re} ( \nabla |s|_h^{4K} \cdot \ov{\nabla} |\textrm{Rm}|) - 2A\, \textrm{Re} ( \nabla |s|^{2K}_h \cdot \ov{\nabla} S) \\ \label{Ric3}
&& \mbox{} - (\Delta |s|_h^{4K}) |\textrm{Rm}| - A(\Delta |s|_h^{2K}) S.
\end{eqnarray}
But at the maximum of $Q$ we have
\begin{equation} \label{Ric4}
\ov{\nabla} |\textrm{Rm}| = - |\textrm{Rm}| \, |s|^{-4K}_h \ov{\nabla} |s|^{4K}_h - A |s|^{-2K}_h \ov{\nabla} S  - A |s|^{-4K}_h S\,  \ov{\nabla} |s|_h^{2K}.
\end{equation}
Combining (\ref{1Rm1}), (\ref{Ric2}), (\ref{Ric3}), (\ref{Ric4}) we have, at a maximum point of $Q$,
\begin{eqnarray} \nonumber
\left( \ddt{}- \Delta \right) Q & \le & C |s|^{4K} |\textrm{Rm}|^2 -  A|s|_h^{2K} (|\overline{\nabla}(\nabla H \, H^{-1})|^2 + |\nabla(\nabla H \,  H^{-1})|^2)  \\
&& \mbox{}   + C |s|_h^{3K} |\textrm{Rm}| + CA |s|^{3K/2}_h | \nabla S| + C(1+A).
\end{eqnarray}
Using (\ref{S3}) and the estimate
\begin{equation}
|\textrm{Rm}|^2 \le C |s|_h^{-K} (|\overline{\nabla}(\nabla H \, H^{-1})|^2 + |\nabla(\nabla H \,  H^{-1})|^2),
\end{equation}
we see that choosing $A$ sufficiently large we obtain, at a maximum point of $Q$,
\begin{equation}
|\textrm{Rm}| \le \frac{C}{|s|_h^{2K}},
\end{equation}
and the case $m=0$ then follows from the maximum principle.

For general $m$ we assume inductively that for $\alpha_m$ and $C$ sufficiently large,
\begin{equation} \label{inductm}
\sum_{i=0}^m | \nablaR^i \textrm{Rm}|^2 \le \frac{C}{|s|_h^{2\alpha_m}}.
\end{equation}
We consider the quantity $Q_{m+1}= |s|_h^{2\alpha_{m+1}} | \nablaR^{m+1} \textrm{Rm}|^2 + A |s|_h^{2\alpha_m} |\nablaR^m \textrm{Rm}|^2$ for $\alpha_{m+1} >> \alpha_m$ and $A$ to be determined later.  Compute using (\ref{evolveRm}) and (\ref{inductm}),
\begin{eqnarray} \nonumber
\left( \ddt{} - \Delta \right) Q_{m+1} & \le & |s|_h^{2\alpha_{m+1}}(-2|\nablaR^{m+2} \textrm{Rm}|^2 + C |s|^{-2\alpha_m}_h |\nablaR^{m+1} \textrm{Rm}| ) \\ \nonumber
&& \mbox{} +  A |s|_h^{2\alpha_m} ( - 2 |\nablaR^{m+1} \textrm{Rm}|^2 + C|s|_h^{-3\alpha_m} ) \\ \nonumber
&& \mbox{} - 2 \textrm{Re} ( \nabla |s|_h^{2\alpha_{m+1}} \cdot \ov{\nabla} | \nablaR^{m+1} \textrm{Rm} |^2 ) - 2 A \textrm{Re} ( \nabla |s|_h^{2\alpha_m} \cdot \ov{\nabla} | \nablaR^m \textrm{Rm}|^2 ) \\
&& \mbox{} - (\Delta |s|_h^{2\alpha_{m+1}}) |\nablaR^{m+1} \textrm{Rm}|^2 - A (\Delta |s|_h^{2\alpha_m}) | \nablaR^m \textrm{Rm} |^2. \label{Qm}
\end{eqnarray}
But we have
\begin{equation}
\left| 2 \textrm{Re}(  \nabla |s|_h^{2\alpha_{m+1}} \cdot \ov{\nabla} | \nablaR^{m+1} \textrm{Rm} |^2) \right|   \le C \alpha_{m+1}^2 | \nablaR^{m+1} \textrm{Rm}|^2 |s|_h^{2\alpha_{m}} +  |s|_h^{2\alpha_{m+1}} | \nablaR^{m+2} \textrm{Rm}|^2, \quad  \label{Qm2}
\end{equation}
and
\begin{equation}
 \left| 2 A \textrm{Re}(  \nabla |s|_h^{2\alpha_m} \cdot \ov{\nabla} | \nablaR^m \textrm{Rm}|^2 ) \right|  \label{Qm3}
\le  A C \alpha_m^2 | \nablaR^m \textrm{Rm}|^2 |s|_h^{2\alpha_m -4} + \frac{A}{2}
 | \nablaR^{m+1} \textrm{Rm}|^2 |s|_h^{2\alpha_m},
\end{equation}
and, if $A$ is sufficiently larger than $\alpha_{m+1}$,
\begin{equation} \label{Qm4}
\left| (\Delta |s|_h^{2\alpha_{m+1}})   |\nablaR^{m+1} \textrm{Rm}|^2 \right|  \le \frac{A}{2} | \nablaR^{m+1} \textrm{Rm}|^2 |s|_h^{2\alpha_m}.
\end{equation}

Then from (\ref{inductm}), (\ref{Qm}), (\ref{Qm2}), (\ref{Qm3}), (\ref{Qm4}) choosing $\alpha_{m+1}$ and then $A$ sufficiently large, we see that at a maximum point of $Q_{m+1}$,
\begin{equation}
| \nablaR^{m+1} \textrm{Rm} |^2 \le \frac{C}{| s|_h^{5\alpha_m}}.
\end{equation}
since we may assume $2a_{m+1}>5\alpha_m$, the inductive step follows after applying the maximum principle to $Q_{m+1}$.
\qed
\end{proof}

Combining the above estimates with local elliptic estimates we have the following corollary.

\begin{corollary} \label{cormetricbd} For each integer $m \ge 0$ there exist $C_m, \alpha_m>0$ such that  for $t\in [0, T)$,
\begin{equation} \label{metricbound}
|(\nabla_{\mathbb{R}}^0)^m g(t)   |_{g_0} \leq \frac{C_m}{|s|_h^{2\alpha_m}},
\end{equation}
where $\nabla_{\mathbb{R}}^0$ denotes the real covariant derivative with respect to the fixed metric $g_0$.
\end{corollary}
\begin{proof}
We work in a fixed coordinate neighborhood of $y_0$ in $Y$.  
  We denote by $L$ the elliptic operator $L = \sum_k \partial_k \partial_{\ov{k}}$.  Let $\alpha>0$ be a sufficiently large constant to be determined later and compute for fixed $i,j$,
\begin{equation} \label{eqnLE}
L(|s|_h^{2\alpha} g_{i \ov{j}}) = - |s|_h^{2\alpha} \sum_k R_{k \ov{k} i \ov{j}} + E,
\end{equation}
where $E$ is an expression involving the metric $g$ only up to first derivatives, and large powers of $|s|_h$.  From Propositions  \ref{propS} and \ref{propRm}, the right hand side of (\ref{eqnLE}) is uniformly bounded if $\alpha$ is chosen to be sufficiently large.  Applying standard elliptic estimates for Sobolev spaces, we obtain bounds on $|s|_h^{2\alpha} g_{i \ov{j}}$ in $L^p_2$ for any $p$.  By the Sobolev embedding theorem this gives us a uniform estimate on $|s|_h^{2\alpha} g_{i \ov{j}}$ in the $C^{1+\gamma}$ norm for some $\gamma \in (0,1)$.  Here we are taking norms with respect to the Euclidean metric in our coordinate patch.  Replace $\alpha$ in (\ref{eqnLE}) with $\tilde{\alpha} >> \alpha$ and we see that the right hand side of  is bounded in $C^{\gamma}$.  Applying Schauder estimates we obtain bounds on $|s|_h^{2\tilde{\alpha}} g_{i \ov{j}}$ in $C^{2+\gamma}$.  Applying a bootstrap argument completes the proof of the corollary.  \qed
\end{proof}


\section{Continuing the K\"ahler-Ricci flow} \label{sectionafterT}

In this section, we will show how to continue the K\"ahler-Ricci flow past time $T$ on the manifold $Y$, thus completing the proof of Theorem \ref{thmblowup}.
  We use some of the techniques developed in \cite{SoT3}. 
  
  Under the assumptions of Theorem \ref{thmblowup}, we have seen that  $(X, g(t))$ converges to $(Y, d_T)$ in the sense of Gromov-Hausdorff as $t\rightarrow T^-$.  We explain how one can continue the K\"ahler-Ricci flow through the singularity, in the manner of \cite{SoT3} (see also the recent paper \cite{SzT}).  We replace $X$ with the manifold $Y$ at the singular time $T$.
Again, we assume for simplicity that we have only one exceptional divisor $E$.

Write $\hat{\omega}_T =  \pi^* \omega_Y$, where $\omega_Y$ is the smooth K\"ahler metric on $Y$.  Then
from Lemma \ref{Linfinity}, there is a closed positive (1,1) current $\omega_T$ and a bounded function $\varphi_T$  with
\begin{equation}
\omega_T  = \hat{\omega}_T + \frac{\sqrt{-1}}{2\pi} \partial \ov{\partial} \varphi_T \ge 0.
\end{equation}
Moreover from Lemma \ref{lemmack},  $\varphi(t)$ converges  to $\varphi_T$ pointwise on $X$ and smoothly on compact subsets of $X \setminus E$.   We have the following lemma.

\begin{lemma} We have
$$\varphi_T|_E = \emph{constant}.$$
Hence there exists a bounded function $\psi_T$ on $Y$, which is smooth on $Y \setminus \{ y_0 \}$,  with $\varphi_T = \pi^* \psi_T$.
\end{lemma}

\begin{proof}
Since $\hat{\omega}_T|_E =0$, we have
\begin{equation}
\frac{\sqrt{-1}}{2\pi} \left( \partial \ov{\partial} \varphi_T\right)|_E = \omega_T|_E \ge 0,
\end{equation}
and thus $\varphi_T$ must be constant on $E$.
  The existence of $\psi_T$ follows immediately from the properties of the blow-down map $\pi$. \qed
\end{proof}

We now define a closed positive (1,1) current $\omega'$ on $Y$ by
\begin{equation}\label{omegaprime}
\omega' =  \omega_Y +  \frac{\sqrt{-1}}{2\pi} \partial \ov{\partial} \psi_T \ge 0.
\end{equation}
Thus $\omega'$ is the push-down of the current $\omega_T$ to $Y$ and is smooth on  $Y \setminus \{ y_0 \}$.

\begin{lemma}  \label{lemmaLp}
There exists  $p>1$ such that
$\displaystyle{\frac{{\omega'}^n}{ \omega^n_Y}  \in L^p(Y)}.$
It follows that $\psi_T$ is continuous on $Y$.
\end{lemma}
\begin{proof} Since $\psi_T$ is bounded and smooth away from $y_0$,  $\omega'^n/\omega_Y^n$ has no mass at $y_0$ and hence is in $L^1(Y)$ .  Here we are using the fact that if $u$ is a bounded plurisubharmonic function then $(\ddbar u)^n$ takes no mass on pluripolar sets (see for example \cite{BT, B, Kol3}).   Applying part (iii) of Lemma \ref{Linfinity},
\begin{equation}
\int_{Y \setminus \{ y_0\}}  \left( \frac{{\omega'}^n}{ \omega^n_Y} \right)^p \omega_Y^n = \int_{X \setminus E} \left( \frac{\omega_T^n}{ (\pi^*\omega_Y)^n} \right)^p (\pi^*\omega_Y)^n \le C \int_{X \setminus E} \left( \frac{\Omega}{ (\pi^*\omega_Y)^n} \right)^{p-1} \Omega \le C,
\end{equation}
as long as $p-1>0$ is sufficiently small.   Hence $\omega'^n/\omega^n_Y$ belongs to $L^p(Y)$.

The fact that $\psi_T$ is continuous follows immediately from a theorem of Kolodziej (see Section 2.4 of \cite{Kol1}).
\qed
\end{proof}

Following \cite{SoT3}, we construct a solution of the K\"ahler-Ricci flow on $Y$ starting at the (possibly singular) metric $\omega'$.   Fix a smooth closed (1,1) form $\chi \in c_1(K_Y)$.  Then there exists $T'>T$ such that  for $t$ in $[T,T']$ the closed $(1,1)$ form
\begin{equation} \label{hatomegatY}
 \hat{\omega}_{t,Y} :=  \omega_Y + (t-T) \chi
\end{equation}
is K\"ahler.  Note that $T'$ is strictly less than the maximal time $T_Y$  in part (iv) of the definition of canonical surgical contraction.
We will use the metrics $\hat{\omega}_{t, Y}$ as our reference metrics for the K\"ahler-Ricci flow as it continues on $Y$.   For simplicity of notation, we will drop the subscript $Y$ and write $\hat{\omega}_{t, Y}$ as $\hat{\omega}_t$.  Fix a smooth volume form $\Omega_Y$ on $Y$ satisfying, for $t \in [T, T']$,
\begin{equation} \label{OmegaYdefn}
\ddbar \log \Omega_Y = \ddt{} \hat{\omega}_t = \chi \in c_1(K_Y).
\end{equation}
We will now construct a family of functions $\psi_{T, \ve}$ on $Y$ which converge to $\psi_T$, using the method of \cite{SoT3}.   For $\ve>0$ sufficiently small, and $K$ fixed and sufficiently large, define a family of volume forms $\Omega_{\ve}$ on $Y$ by
\begin{equation} \label{Omegaepsilon}
\Omega_{\ve} = (\pi|_{X\setminus E}^{-1})^* \left( \frac{ |s|_h^{2K} \omega^n(T-\ve)} {\ve + |s|_h^{2K}} \right)+ \ve \Omega_Y \quad \textrm{on } Y \setminus \{ y_0 \},
\end{equation}
and  $\Omega_{\ve}|_{y_0} = \ve \Omega_Y|_{y_0}$.  It follows from the definition of the map $\pi$ that, after choosing $K$ sufficiently large, the volume form  $\Omega_{\ve}$ lies in $C^{\ell}$ for a fixed large constant $\ell$.  Moreover, $\Omega_{\ve}$ converges to $\omega'^n$ in $C^{\infty}$ on compact subsets of $Y \setminus \{ y_0 \}$ as $\ve$ tends to zero.  We define our functions $\psi_{T, \ve}$ to be solutions of the complex Monge-Amp\`ere equations:
\begin{equation} \label{mapsi}
( \omega_Y + \ddbar \psi_{T, \ve})^n = C_{\ve} \Omega_{\ve}, ~~~~~\sup_Y (\psi_{T, \ve} - \psi_T) = \sup_Y (\psi_T - \psi_{T,\ve}),
\end{equation}
where the constants $C_{\ve}\in \mathbb{R}$ are chosen so that $C_{\ve} \int_Y \Omega_{\ve} = \int_Y  \omega_Y^n.$  Solutions $\psi_{T, \ve}$ to (\ref{mapsi}) exist by Yau's theorem \cite{Y1}, are unique by the well-known result of Calabi, and lie in $C^k(Y) \cap C^{\infty}(Y \setminus \{y_0 \})$ due to the regularity of $\Omega_{\ve}$.  Note that we are free to raise $k$ by increasing $\ell$ and $K$.

By the definition of $\Omega_{\ve}$ we have
\begin{equation} \label{kol}
 \|  \frac{C_\ve \Omega_\ve}{\Omega_Y} - \frac{\omega'^n}{\Omega_Y}  \|_{L^1(Y)} \rightarrow 0, \quad \textrm{as } \ve \rightarrow 0.
 \end{equation}
From Lemma \ref{lemmaLp}  and  Kolodziej's stability theorem  \cite{Kol2}, we have
\begin{equation} \label{psiconvergence}
 \| \psi_{T,\ve} - \psi_T\|_{L^{\infty}(Y)} \rightarrow 0, \quad \textrm{as } \ve \rightarrow 0.
\end{equation}
Now let $\varphi_{\ve}$ be solutions of the parabolic complex Monge-Amp\`ere equations
 \begin{equation}
 \ddt{\varphi_{\ve}}= \log \frac{(\hat{\omega}_t + \ddbar\varphi_{\ve})^n}{\Omega_Y}, \qquad \varphi_{\ve} |_{t=T} = \psi_{T, \epsilon},
 \end{equation}
 for $t \in [T, T']$.
 Then we have the following.

\begin{proposition} \label{propst}
There exists a function $\varphi$ in $C^0([T,T'] \times Y) \cap C^{\infty}((T,T'] \times Y)$ such that
\begin{enumerate}
\item[(i)]  $\varphi_{\ve} \rightarrow \varphi$ in $L^{\infty}([T,T'] \times Y)$.
\item[(ii)]  The convergence $\varphi_{\ve} \rightarrow \varphi$ is $C^{\infty}$ on compact subsets of $(T, T']\times Y$.
\item[(iii)]  $\varphi$ is the unique solution of
\begin{equation} \label{pmast}
\varphi|_{t=T} = \psi_T, \quad \ddt{\varphi}= \log \frac{(\hat{\omega}_t + \ddbar\varphi)^n}{\Omega_Y} \quad \textrm{for } t \in (T, T'],
 \end{equation}
in the space $C^0([T,T'] \times Y) \cap C^{\infty}((T,T'] \times Y)$.
\end{enumerate}
\end{proposition}
\begin{proof}  This is a special case of a general result proved in \cite{SoT3} and so we omit the proof.  \qed
\end{proof}

In order to prove sharper bounds on $\varphi_{\ve}$ and $\varphi$, we will need estimates on $\psi_{T, \varepsilon}$.  Write 
\begin{equation}
\omega_{T, \ve} := T\omega_Y + \ddbar \psi_{T, \ve},
\end{equation} 
and write $g_{T, \ve}$ for the corresponding K\"ahler metric.  For convenience we will write the function $(\pi|_{X\setminus E}^{-1})^* (|s|_h^2)$ on $Y\setminus \{ y_0\}$ as $|s|_h^2$. Making use of the estimates of Section \ref{sectionhigher}, we prove:

\begin{lemma} \label{lemmabdpsi}  
There exist positive constants $C, \alpha$, independent of $\ve$, such that on $Y \setminus \{ y_0 \}$,
\begin{equation} \label{otep1}
\displaystyle{\frac{|s|_h^{2\alpha}}{C} \omega_Y \le \omega_{T, \ve} \le \frac{C}{|s|_h^{2\alpha}} \omega_Y}.
\end{equation}
Fix a large positive integer $N$.  Then for each integer $0 \le m \le N$ there exist $C_m, \alpha_m>0$ such that
\begin{equation} \label{metricbound2}
|(\nabla_{\mathbb{R}}^Y)^m g_{T, \ve}   |_{g_Y} \leq \frac{C_m}{|s|_h^{2\alpha_m}},
\end{equation}
where $\nabla_{\mathbb{R}}^Y$ denotes the real covariant derivative with respect to the fixed metric $g_Y$.
\end{lemma}

\begin{proof}
Write $\tilde{\omega} = \omega_{T, \ve}$ and $\hat{\omega} = \hat{\omega}_T = T \omega_Y$ and observe that $\tilde{\omega}$ satisfies the equation (\ref{mapsi}), which can be written
\begin{equation}
\tilde{\omega}^n = (\hat{\omega} + \ddbar \psi_{T,\ve})^n = e^{F_{\ve}} \hat{\omega}^n, \quad \textrm{where }\  F_{\ve}=  \log \left( \frac{C_{\ve}\Omega_{\ve}}{\hat{\omega}^n} \right).
\end{equation}
We define
\begin{equation}
Q = \log ((\tr{\hat{\omega}}{\tilde{\omega}}) |s|_h^{2\alpha}) - B \psi_{T, \ve},
\end{equation}
where $B$ is  a large constant to be determined later.
Observe that  $Q(x)$ tends to negative infinity as $x$ tends to $y_0$.  Moreover, $\psi_{T, \ve}$ is uniformly bounded by (\ref{psiconvergence}).
Following the well-known second order estimate of Yau \cite{Y1}, we compute at a point in $Y \setminus \{y_0\}$,
\begin{equation} \label{lapL}
\Delta_{\tilde{\omega}} Q \ge - C\tr{\tilde{\omega}}{\hat{\omega}} - C + \Delta_{\hat{\omega}} F_{\ve} - \alpha \tr{\tilde{\omega}}{(R(h))} + B \tr{\tilde{\omega}}{\hat{\omega}},
\end{equation}
where $C$ is a uniform constant and we recall that $R(h)$ is the curvature of the Hermitian metric $h$.  From the definition of $\Omega_{\ve}$ together with (\ref{so1}) and Corollary \ref{cormetricbd} we have
\begin{equation} \label{estimateF}
|\Delta_{\hat{\omega}} F_{\ve} | \le \frac{C}{|s|^{2\beta}_h},
\end{equation}
for uniform constants $C$ and $\beta$, which are independent of $\ve$.  Then at a maximum point of $Q$ we have from (\ref{lapL}) and (\ref{estimateF}),
\begin{equation}
\tr{\tilde{\omega}}{\hat{\omega}} \le \frac{C}{|s|_h^{2\beta}},
\end{equation}
by choosing $B$ sufficiently large.
Then at this maximum point we have
\begin{eqnarray} \nonumber
(\tr{\hat{\omega}}{\tilde{\omega}}) |s|_h^{2\alpha} & \le & \frac{1}{(n-1)!}   \left(\tr{\tilde{\omega}}{\hat{\omega}}\right)^{n-1}  \, \left(\frac{\tilde{\omega}^n}{\hat{\omega}^n}\right)|s|_h^{2\alpha} \\
& \le & C |s|_h^{2(\alpha-\beta')}
\end{eqnarray}
for some constants $\beta'$ and $C$.  Choosing $\alpha$ sufficiently large we see that $Q$ is uniformly bounded from above and (\ref{otep1}) follows.

For (\ref{metricbound2}), we first observe that by  
 Yau's third order estimate (cf. equation (2.43) of \cite{PSS}) we have for constants $C$ and $\beta$,
\begin{equation}
\Delta_{\tilde{\omega}} S \ge - \frac{C}{|s|_h^{2\beta}} \left( | \nabla_{\tilde{\omega}} \textrm{Ric}(\tilde{\omega})|_{\tilde{\omega}}^2 + S +1\right),
\end{equation}
where $S= | \nabla H_{T, \ve} \, H_{T, \ve}^{-1}|^2$ where $(H_{T, \ve})^i_{\ell} = g_Y^{i \ov{j}} (g_{T, \ve})_{\ell \ov{j}}$.  Applying Corollary \ref{cormetricbd} again, we have 
\begin{equation} \label{derivRicbd}
| \nabla_{\tilde{\omega}} \textrm{Ric}(\tilde{\omega})|_{\tilde{\omega}}^2 \le \frac{C}{|s|_h^{2\beta'}}.
\end{equation}

Hence, after possibly increasing $\beta$ and $C$, we have
\begin{equation}
\Delta_{\tilde{\omega}} S \ge - \frac{C}{|s|_h^{2\beta}} \left( S +1\right),
\end{equation}
and from Yau's second order estimate,
\begin{equation}
\Delta_{\tilde{\omega}} \tr{\tilde{\omega}}{\hat{\omega}} \ge \frac{|s|_h^{2\beta} S}{C} - \frac{C}{|s|_h^{2\beta}}.
\end{equation}
Applying the maximum principle to the quantity $|s|^{2\alpha_1}_h S + A |s|^{2\alpha_0} \tr{\tilde{\omega}}{\hat{\omega}}$, where we choose $\alpha_0$, $\alpha_1$ and $A$ sufficiently large, we obtain
$|S| \le \frac{C}{|s|^{2\alpha}}.$

It is now a straightforward matter to obtain (\ref{metricbound2}) using a similar method to that of Section \ref{sectionhigher}.  Indeed, we can obtain estimates on all  covariant derivatives of the curvature of $\omega_{T, \ve}$ by an elliptic analogue of Proposition \ref{propRm} (for the sake of brevity, we omit the proof).   Then (\ref{metricbound2}) follows by the same argument as in Corollary \ref{cormetricbd}.
 \qed
\end{proof}

We now prove estimates on the evolving metric on $Y$ after time $T$.
Define for $t \in [T, T']$,
\begin{equation} \label{omegave}
\omega_{\ve}(t) = \hat{\omega}_t + \ddbar\varphi_{\ve}.
\end{equation}
We first obtain an estimate on the volume form of $\omega_{\ve}=\omega_{\ve}(t)$ for $t \in [T, T']$.

\begin{lemma} \label{lemmavomegave} There exist positive constants $\alpha$ and $C$, independent of $\ve$, such that
\begin{equation} \label{vomegave}
\frac{ \omega_{\ve}^n }{\Omega_Y} \leq \frac{C}{  |s|_h^{2\alpha}}
\end{equation}
on $[T, T'] \times (Y \setminus \{ y_0 \})$.
\end{lemma}
\begin{proof}  We use the maximum principle.  Consider the quantity
\begin{equation} \label{Q1}
Q = \dot{\varphi}_{\ve} + \alpha \log |s|_h^2 - A \varphi_{\ve}
\end{equation}
on $[T, T'] \times (Y \setminus \{ y_0\})$ for positive constants $\alpha$ and $A$ to be determined later.  Note that by Proposition \ref{propst}, $\varphi_{\ve}$ is uniformly bounded on $[T, T'] \times Y$, independent of $\ve$. We have
\begin{equation}
Q|_{t=T} = \frac{C_{\ve} \Omega_{\ve}}{\Omega_Y} + \alpha \log |s|_h^2 - A \psi_{T, \ve}.
\end{equation}
From the definition of $\Omega_{\ve}$ and Lemma \ref{lemmaest1} we see that $Q$ is uniformly bounded from above, independent of $\ve$, at time $T$ if we choose $\alpha$ sufficiently large.  Moreover, the quantity $Q(x,t)$ tends to negative infinity as $x$ tends to $y_0$, for any $t \in [T, T']$.

Compute
\begin{eqnarray} \nonumber
\left( \ddt{} - \Delta_{\omega_{\ve}} \right) Q & = & \textrm{tr}_{\omega_{\ve}} ( \chi + \alpha R(h) + A(\omega_{\ve} - \hat{\omega}_{t})) - A \dot{\varphi}_{\ve} \\
& = & \textrm{tr}_{\omega_{\ve}}( \chi  + \alpha R(h) - A \hat{\omega}_t) + An - A(Q - \alpha \log |s|_h^2 + A \varphi_{\ve}).
\end{eqnarray}
Choosing $A$  large enough so that $\chi + \alpha R(h) - A \hat{\omega}_t <0$ and using the fact that $\varphi_{\ve}$ is uniformly bounded, we obtain
\begin{equation}
\left( \ddt{} - \Delta_{\omega_{\ve}} \right) Q  \le  C' - AQ,
\end{equation}
for some uniform constant $C'$,
and a standard maximum principle argument shows that $Q$ is uniformly bounded from above for $t$ in  $[T,T']$.  The estimate (\ref{vomegave}) follows immediately. \qed
\end{proof}

We will make use of the volume bound of Lemma \ref{lemmavomegave}, as well as the estimates of Lemma \ref{lemmabdpsi}  to bound the metric $\omega=\omega(t)$ from above.

\begin{lemma}  \label{afterT}
There exist positive constants $\alpha$ and $C$, independent of $\ve$, such that on $[T, T'] \times (Y \setminus \{ y_0 \})$, we have
\begin{equation} \label{oe1}
\displaystyle{\frac{|s|_h^{2\alpha}}{C} \omega_Y \le \omega_{\ve} \le \frac{C}{|s|_h^{2\alpha}} \omega_Y}.
\end{equation}
Fix a large positive integer $N$.  Then for each integer $0 \le m \le N$ there exist $C_m, \alpha_m>0$ such that  for $t \in [T, T']$,
\begin{equation} \label{metricbound3}
|(\nabla_{\mathbb{R}}^Y)^m g_{\ve}   |_{g_Y} \leq \frac{C_m}{|s|_h^{2\alpha_m}}.
\end{equation}
\end{lemma}

\begin{proof}
Define
\begin{equation} \label{Q2}
Q = \log \left( ( \tr{\omega_Y}{\omega_{\ve}})  |s|_h^{2\alpha} \right) - A \varphi_{\ve},
\end{equation}
for positive constants $\alpha$ and $A$ to be determined later.  Observe that the quantity $Q(x,t)$ tends to negative infinity as $x$ tends to $y_0$, for any $t \in [T, T']$. Moreover
from Lemma \ref{lemmabdpsi} we see that, choosing $\alpha$ sufficiently large, $Q$ is uniformly bounded from above at time $T$.

Compute, using (\ref{tr1}),
\begin{equation}
\left( \ddt{} - \Delta_{\omega_{\ve}} \right) Q \le  \textrm{tr}_{\omega_{\ve}} (C \omega_Y + \alpha R(h) + A(\omega_{\ve}-\hat{\omega}_t)) - A\dot{\varphi}_{\ve}
\end{equation}
for some uniform constant $C$.  Hence, choosing $A$ sufficiently large, we obtain
\begin{equation}
\left( \ddt{} - \Delta_{\omega_{\ve}} \right) Q  \le  - \tr{\omega_{\ve}}{\omega_Y} - A \log \frac{\omega_{\ve}^n}{\Omega_Y} + C,
\end{equation}
and thus at a point where $Q$ achieves its maximum we obtain
\begin{equation}
 \tr{\omega_{\ve}}{\omega_Y} + A \log \frac{\omega_{\ve}^n}{\Omega_Y} \le C.
\end{equation}
Then using the argument of Lemma \ref{lemmaest1} together with the volume bound of  Lemma \ref{lemmavomegave}, we see that, if $\alpha$ is sufficiently large, $Q$ is uniformly bounded from above independent of $\ve$.  Since $\varphi_{\ve}$ is uniformly bounded, the upper bound of $\omega_{\ve}$  follows.  The lower bound of $\omega_{\ve}$ follows from an argument similar to the proof of Lemma \ref{lemmack}.  

For (\ref{metricbound3}), we first bound
$S_{\ve} = | \nabla H_{\ve} \, H^{-1}_{\ve} |^2$, where  $(H_{\ve})^i_{\, \ell} = g_Y^{i \ov{j}} (g_{\ve})_{\ell \ov{j}}$.  The estimate $\displaystyle{|S_{\ve}| \le \frac{C}{|s|
_h^{2\alpha}}}$ follows by a similar argument to that of Proposition \ref{propS}, making use of the bound on $S$ at time $T$ provided by Lemma \ref{lemmabdpsi}.

We bound the curvature of $\omega_{\ve}$ and its covariant derivatives  by applying the argument of Proposition \ref{propRm} together with the estimates of Lemma \ref{lemmabdpsi}.  Then (\ref{metricbound3}) follows from the same proof as Corollary \ref{cormetricbd}.
\qed
\end{proof}

Recall from Proposition \ref{propst} that $\varphi(t)$ for $t \in [T,T']$ is the limit of $\varphi_{\ve}$ as $\ve$ tends to zero.  
The metric $\omega = \hat{\omega}_t + \ddbar \varphi$ for $t \in [T,T']$ is a solution of the K\"ahler-Ricci flow on $Y$:
\begin{equation}
\ddt{} \omega = - \textrm{Ric}(\omega), \quad \textrm{for } \ t \in (T, T'],
\end{equation}
and has $\omega|_{t=T} = \omega'$.   Lemma \ref{afterT} gives estimates on $\omega(t)$ for $t \in [T,T']$ on  $Y \setminus \{ y_0\}$, and Corollary \ref{cormetricbd} gives us estimates on $\omega(t)$ for $t \in [0,T)$ on $X \setminus E$.

We can now prove that the K\"ahler-Ricci flow can be smoothly connected at time $T$ between $[0, T) \times X$ and $(T,T'] \times Y$, outside $T \times \{ y_0\} \cong T \times E$ via the map $\pi$.  To make this precise, define
\begin{equation}
Z = ([0, T) \times X) \cup (T \times Y \setminus \{ y_0 \}) \cup ((T, T'] \times Y).
\end{equation}
Consider a family of metrics $\omega(t,x)$, for $(t,x) \in Z$.  We will define what it means to say that $\omega$ is smooth on $Z$.
  If $t \in [0,T)$ or $t \in (T, T']$ then we demand $\omega$ to be smooth at $t$ in the usual sense, in $X$ or $Y$ respectively.  On the other hand, if $(t,x) = (T,x) \in T \times Y \setminus \{ y_0 \} \cong T \times X \setminus E$ then we take a small neighborhood $U$ of $x$ in $X \setminus E$ and, via the map $\pi$, consider $\omega$ as a metric on $(T-\delta, T+\delta) \times U$, for some $\delta>0$.  We say $\omega$ is smooth at $(T,x)$ if $\omega$ is smooth at $(T,x)$ in $(T-\delta, T+\delta) \times U$.  In the same way, we can define what it means for $\omega(t)$ to satisfy a PDE (such as the K\"ahler-Ricci flow) at an arbitrary point of $Z$.

\begin{theorem}
The solution $\omega=\omega(t)$ is a smooth solution of the K\"ahler-Ricci flow in the space-time region $Z$.
\end{theorem}
\begin{proof} We only need to check that $\omega$ satisfies the K\"ahler-Ricci flow and is smooth at time $T$ in the sense above.  This is immediate from  Corollary \ref{cormetricbd} and Lemma \ref{afterT}. \qed \end{proof}

This completes the proof that part (iv) in the definition of canonical surgical contraction holds under the assumptions of Theorem \ref{thmblowup}.


\section{Backwards Gromov-Hausdorff convergence} \label{sectionproof}

In this section, we complete the proof of Theorem \ref{thmblowup}.  It remains to show that $(Y, \omega(t))$ converges in the Gromov-Hausdorff sense to $(Y, d_T)$ as $t\rightarrow T^+$.  We have an analogue of Lemma \ref{lemmaest1}.

\begin{proposition} \label{extendest}
There exist $\delta>0$ and a uniform constant $C$ such that the following estimates hold for $\omega= \omega(t)$ with $t \in [T, T']$,
\begin{enumerate}
\item[(i)] $\displaystyle{
 \omega \le C \frac{\omega_Y}{(\pi|_{X\setminus E}^{-1})^* |s|_{h}^{2}} }$,
\item[(ii)] $\displaystyle{
 \omega \le C (\pi|_{X \setminus E}^{-1})^* \left( \frac{\omega_0 }{ |s|_{h}^{2(1- \delta)}} \right)}$,
\end{enumerate}
where we recall that $\omega_0$ is the initial metric on $X$.
\end{proposition}

\begin{proof}  Since the proof involves is similar to the arguments given in the earlier sections, we give here just an outline of the proof.  For more details, we refer the reader to \cite{SW2}, where this result is proved in a more general setting.

We prove the existence of a solution $\tilde{\rho}$ of the parabolic complex Monge-Amp\`ere equation
\begin{equation} \label{6maY}
\ddt{\tilde{\rho}} = \log \frac{ ( \hat{\omega}_{t,Y} + \ddbar \tilde{\rho} )^n }{ \Omega_Y}, ~~~\tilde{\rho} |_{t=T}= \psi_T,
\end{equation}
on $[T, T'] \times (Y \setminus \{ y_0 \})$ by considering a family of Monge-Amp\`ere flows on $X$.   First, let $f_{\ve}$ be a family of positive smooth functions on $Y$ of the form
\begin{equation}
f_{\ve} (z) = (\ve + r^2)^{n-1}, \quad \textrm{on } D,
\end{equation}
which converge to a function $f$ which is of the form $f(z) = r^{2(n-1)}$ on the unit ball $D$ and is positive on $X \setminus D$.  By the definition of the blow-down map there is a smooth volume form $\Omega_X$ on $X$ with $\pi^* \Omega_Y = (\pi^* f) \Omega_X$.

 Note that if $\ve>0$ is sufficiently small then $\hat{\omega}_{t,Y} - \frac{\ve}{T} \omega_Y$ is K\"ahler on $Y$ for  $t \in [T,T']$.  We consider for $\ve>0$ sufficiently small
 the following family of  Monge-Amp\`ere flows on X:
\begin{equation} \label{maX1}
\ddt{\rho_{\ve}} = \log \frac{(  \pi^* (\hat{\omega}_{t,Y}- \frac{\ve}{T} \omega_Y) + \frac{\ve}{T} \omega_0  + \ddbar \rho_{\ve})^n}{(\pi^*f_{\ve}) \Omega_X }, ~~~ \rho_{\ve} |_{t=T} = \varphi(T-\ve).
\end{equation}
Observe that at $t=T$, the (1,1) form $\pi^* (\hat{\omega}_{t,Y}- \frac{\ve}{T} \omega_Y) + \frac{\ve}{T} \omega_0$ is equal to $\hat{\omega}_{T-\ve}$.

A standard argument gives an uniform  bound for $| \rho_{\ve}|$ independent of $\ve$.  Making minor modifications to the proof of Lemma \ref{lemmaest1} to deal with the extra terms coming from $f_{\ve}$ we obtain the estimates 
\begin{equation}
\omega_{\ve} \le \frac{C}{|s|^2_h} \pi^* \omega_Y \quad \textrm{and} \quad \omega_{\ve} \le \frac{C}{|s|_h^{2(1-\delta)}} \omega_0, \quad \textrm{ on } (X\setminus E) \times [T, T']
\end{equation}
for $\omega_{\ve} =  \pi^* (\hat{\omega}_{t,Y}- \frac{\ve}{T} \omega_Y) + \frac{\ve}{T} \omega_0 +  \ddbar \rho_{\ve}$
and $C^{\infty}$ estimates for $\omega_{\ve}$ on compact subsets away from $E$.  Letting $\ve \rightarrow 0$ and pushing forward to $Y$ we get a smooth solution $\tilde{\rho}$ of (\ref{6maY}) on $[T, T'] \times (Y \setminus \{ y_0 \})$ with $\hat{\omega}_{t, Y} + \ddbar \tilde{\rho}$ satisfying the estimates (i) and (ii).  

On the other hand, $\tilde{\rho}$ is equal to the solution $\varphi$ on $Y$ we constructed in Proposition \ref{propst} (this follows from a slightly more general uniqueness theorem which can be proved similarly).  Hence (i) and (ii) hold for $\omega(t)$ as required.
\qed
\end{proof}

It then follows by the arguments of Section \ref{sectionGH} that:

\begin{theorem}
$(Y, \omega(t))$ converges in the Gromov-Hausdorff sense to $(Y, d_T)$ as $t\rightarrow T^+$.
\end{theorem}

This completes the proof of Theorem \ref{thmblowup}.

\section{The K\"ahler-Ricci flow for algebraic surfaces}  \label{sectionapply}

In this subsection, we give the proof of Theorem \ref{theoremclass}.  The argument is purely algebraic, and follows a line of reasoning well-known to experts in the field.  However, for the sake of completeness, we provide here some details of the proof.
We first recall some basic definitions from algebraic geometry.

Let $L$ be a holomorphic line bundle over an algebraic variety $X$ of dimension $n$.  We say  that $L$ is \emph{nef} if $L\cdot C := \int_C c_1(L)  \ge 0$ for all irreducible curves $C$ in $X$.   If $L$ is nef then we say that $L$ is \emph{big} if $L^n := \int_X c_1(L)^n>0$.
Finally $L$ is \emph{semi-ample} if there exists an integer $m>0$ such that the line bundle $L^m$ is globally generated (meaning that for each $x \in X$, there exists a holomorphic section $s$ of $L^m$ with $s(x)\neq 0$).
The definitions of nef and big naturally extend to formal linear $\mathbb{R}$ combinations of line bundles.

We impose the condition now that the cohomology class of $[\omega_0]$ is rational.  Then  we may assume without loss of generality that $[\omega_0] = c_1(L)$ where $L$ is an ample line bundle over $X$.
Define  $\tilde{T}$ by
\begin{equation}
\tilde{T} = \sup \{ t \in \mathbb{R} \ | \ L+tK_X  \ \textrm{is nef} \, \}.
\end{equation}
Then we have the following lemma.

\begin{lemma}  Assume that $\tilde{T}<\infty$.  Then:
\begin{enumerate}
\item[(i)] $\tilde{T}$ is rational.
\item[(ii)] For any integer $r \ge 1$ such that $r\tilde{T} \in \mathbb{Z}$, the line bundle $M=r(L+\tilde{T}K_X)$ is  nef and semi-ample.
\item[(iii)] $\tilde{T}= T$, where $T$ is given by (\ref{T}).
\end{enumerate}
\end{lemma}
\begin{proof}
Part (i) is a direct consequence of the Rationality Theorem of Kawamata and Shokurov (see \cite{KMM, KM}).

The Base-Point-Free Theorem of Kawamata \cite{Ka} states that if $L$ is a nef line bundle over a projective manifold such that $aL-K_X$ is nef and big for some $a>0$ then $L$ is semi-ample.  Let $T$ be given by (\ref{T}).  Then $T \le \tilde{T}$ and $L+(T-\varepsilon)K_X$ is nef and big for $\varepsilon>0$ sufficiently small.
  Since $L+\tilde{T}K_X$ is nef it follows that $M=r(L+\tilde{T}K_X)$ is semi-ample, giving (ii).

Finally we show that $T=\tilde{T}$.    Suppose for a contradiction that $T<\tilde{T}$.  Then for any $T_0 \in (T, \tilde{T})$,  $L+ T_0 K_X$ is nef.  But if $N$ is any nef line bundle and $A$ any ample line bundle then $kN+  A$ is ample for $k \ge 0$ by the Nakai-Moishezon criterion (see for example Proposition 6.2 of \cite{De}).  Thus $(1+\delta) L + T_0 K_X$ is ample for $\delta>0$ sufficiently small.  This contradicts the definition of $T$.
\qed
\end{proof}

We now give the proof of Theorem \ref{theoremclass}.   The behavior of the flow at a singular time $T<\infty$ depends on whether the line bundle $M$ from the lemma above is big or not.

We show that if  $M$ is big then the K\"ahler-Ricci flow performs a canonical surgical contraction at $T$ with respect to a finite number of disjoint exceptional curves.  Since $M$ is big, nef and semi-ample, applying Theorem 2.1.27 in \cite{L} we obtain a map $\pi : X \rightarrow Y$ where $Y$ is a normal  projective variety of complex dimension two.  The exceptional locus of $\pi$ is a sum of irreducible curves $C=\sum_i  C_i$ with $C_i \cdot M=0$ and hence $K_X \cdot C_i<0$.  Since $M^2>0$ we can apply the index theorem (see for example \cite{GH}, page 471) to see that $C_i^2 < 0$, and hence  by the adjunction formula, $C_i$ is a smooth rational curve with $C_i^2 =-1$.  Moreover, the curves $C_i$ are disjoint.  Indeed, suppose distinct irreducible curves $C_1$ and $C_2$ satisfy $C_1 \cdot M =0 = C_2 \cdot M$ and $C_1^2=-1=C_2^2$.  Then by the index theorem again, $(C_1 +C_2)^2 < 0$ and hence $C_1 \cdot C_2=0$.

 Thus $\pi$ is a map blowing down the exceptional curves $C_i$ (see for example Theorem 4.2 of \cite{BHPV}) and $Y$ is smooth. Since $M$ is the pull-back of an ample line bundle over $Y$, we obtain (\ref{assumption1}) for some K\"ahler metric $\omega_Y$ on $Y$.  Thus we can apply Theorem \ref{thmblowup} to see that the K\"ahler-Ricci flow performs a contraction with respect to a finite number of disjoint exceptional curves.  

On the other hand, if the line bundle $M$ is not big then $c_1(M)^2=0$ and hence
\begin{equation} \label{collapse}
\textrm{Vol}_{\omega(t)}X = \int_X \omega(t)^2 \rightarrow c_1(M)^2 =0, \quad \textrm{as} \quad t\rightarrow T^-.
\end{equation}
In this case, we cannot have a canonical surgical contraction.

In this way we obtain a sequence of canonical surgical contractions $g(t)$ on the manifolds $X_0=X, X_1,  X_2, \ldots, X_k$ on the time intervals $[0,T_0), (T_1, T_2), \ldots, (T_{k-1}, T_k)$  of the type described above.  We define  $M_0, \ldots, M_k$ by $M_0=L$ and $M_i = M_{i-1} + T_i K_{X_i}$ for $i=1, \ldots, k$.  Then the surgical contractions end when  either $T_k=\infty$ or  $M_k$ on $X_k$ is not big, in which case we have (\ref{collapse}).  We can only have a finite number of canonical surgical contractions, since with each one the second Betti number strictly decreases.

We now  show that in the case $T_k<\infty$, the manifold $X_k$ is either Fano (that is, $K^{-1}_{X_k}$ is ample) or a ruled surface.   Since $M_k$  is not big we can apply Theorem 2.1.27 in \cite{L} to see that for $m$ sufficiently large, the morphism
\begin{equation}
\pi: X_k \rightarrow Y_m \subset \mathbb{P} (H^0(X_k, M_k^m)),
\end{equation}
has image $Y=Y_m$ either a point or a normal (hence nonsingular) curve and a multiple of $M_k$ over $X_k$ is the pull-back of an ample line bundle on $Y$.  If $Y$ is a point then $K_{X_k}<0$ and thus $X_k$ is Fano.  If $Y$ is a non-singular curve then a generic fiber $C$ of $\pi$ is a smooth curve satisfying $M_k\cdot C =0$ and hence $K_{X_k} \cdot C<0$.  It follows from the adjunction formula  that $C$ is a rational curve.  Hence $X_k$ admits a ruled structure over $Y$.   

If $T_{k}=\infty$ then for some ample line bundle $A$ we have that $A+tK_{X_k}$ is nef for all $t$ and hence $K_{X_k}$ is nef.  
It follows that $X_k$ contains no exceptional  curves of the first kind since from Proposition 2.2 in \cite{BHPV}, such a curve $C$ must have $K_X \cdot C<0$, whereas we have just shown that $K_X$ is nef.    This completes the proof of Theorem \ref{theoremclass}.

\begin{remark} \label{negkod}
In case (ii), since $K_X$ is nef, $X$ must have nonnegative Kodaira dimension \cite{BHPV} and in this case there is a unique minimal model of $X$.
\end{remark}

\bigskip

\bigskip
\noindent
{\bf Acknowledgements} \ The authors express their gratitude to D.H. Phong for his support and advice. The first-named author also thanks G. Tian for many helpful discussions.   In addition, the authors are grateful to G. Sz\'ekelyhidi, V. Tosatti and Y. Zhang for a number of very useful conversations.   The first-named author carried out part of this work while visiting Columbia University, and thanks the department for its kind hospitality.

\bigskip
\bigskip

$^{*}$ Department of Mathematics \\
Rutgers University, Piscataway, NJ 08854\\

$^{\dagger}$ Department of Mathematics \\
University of California San Diego, La Jolla, CA 92093

\end{document}